\documentclass{amsart}
\usepackage{amssymb,graphicx,rotating,multirow,multicol}
\usepackage{amsmath}
\usepackage{amsthm}
\usepackage{amsfonts}
\usepackage[toc,page]{appendix}
\usepackage{caption}
\usepackage{subfigure}
\usepackage{epstopdf}
\usepackage{float}
\usepackage[T1]{fontenc}
\usepackage[utf8]{inputenc}
\usepackage{graphicx}

\swapnumbers
\newtheorem{theorem}{Theorem}[subsection]

\newtheorem{lemma}[theorem]{Lemma}

\newtheorem{cor}[theorem]{Corollary}
\newtheorem{proposition}[theorem]{Proposition}
\theoremstyle{remark}
\newtheorem{remark}[theorem]{Remark}
\numberwithin{equation}{subsection}
\newcommand{\Z}{\ensuremath{\mathbb{Z}}}
\newcommand{\R}{\ensuremath{\mathbb{R}}}

\def\ie{\emph{i.e.}}

{\catcode`\@=11 \gdef\mnote#1{\marginpar{\footnotesize
 \tolerance\@M\spaceskip2.6\p@ plus10\p@ minus.9\p@\rm#1}}}
\def\Dg:{\endgraf{\bf Dg:\enspace}\ignorespaces}

\let\Bbb\mathbb
\let\Cal\mathcal

\def\LC#1#2{\operatorname{LC}^{#1}_{#2}}
\def\QC#1#2{\operatorname{QC}^{#1}_{#2}}
\def\QD#1{\operatorname{Q\D}^{#1}}
\def\idx#1#2{d_{#1,#2}}
\def\D{\Delta}
\def\G{\Gamma}
\def\P{\Cal P}
\def\s{\sigma}
\def\sm{\smallsetminus}

\def\XYZ{\operatorname{XYZ}}
\def\Cr#1#2#3{\operatorname{Cr}_{#1#2#3}}

\newcommand{\be}{\begin{equation}}
\newcommand{\ee}{\end{equation}}
\let\ge\geqslant 
\let\le\leqslant 

\def\Z{\Bbb Z}
\def\R{\Bbb R}

\def\DD{\Bbb D}

\def\Rp#1{\Bbb{RP}^{#1}}
\def\Aut{\operatorname{Aut}}

\makeatletter
\newcommand{\addresseshere}{%
  \enddoc@text\let\enddoc@text\relax
}
\makeatother

\begin{document}

\renewcommand{\arraystretch}{1.2}
\title[Deformation Classification of Typical Configurations of 7 Points]
{Deformation Classification of Typical Configurations of 7 Points in the Real Projective Plane}
\author[S.~Finashin R.A.~Zabun]{Sergey Finashin and Remz\.{i}ye Arzu Zabun }
\address{
 Department of Mathematics, Middle East Tech. University\\
 06800 Ankara Turkey}
\address{
 Department of Mathematics, Gaziantep University\\
 27310 Gaziantep Turkey}

\subjclass[2010]{Primary:14P25. Secondary: 14N20.}

\keywords{Typical configurations of 7 points, deformations, real Aronhold sets}
\date{}
\dedicatory{}

\begin{abstract}
A configuration of $7$ points in $\Rp2$ is called {\it typical} if it has no collinear triples and no coconic sextuples of points.
We show that there exist $14$ deformation classes of such configurations. This yields classification of real Aronhold sets.
\end{abstract}

\maketitle

\rightline
{\vbox{\hsize80mm
\noindent\baselineskip10pt{\it\footnotesize
``This is one of the ways in which the magical number seven has persecuted me.''
\vskip3mm
\noindent
\rm
George A. Miller, \it
The magical number seven, plus or minus two:
some limits of our capacity for processing information
}}}
\vskip5mm

\section{Introduction}
\subsection{Simple configurations of $n\le7$ points}
Projective configurations of points on the plane is a classical subject in algebraic geometry and
its history in the context of linear systems of curves can be traced back
to 18th century (G.\,Cramer, L.\,Euler, etc.).
In modern times, projective configurations are studied both from
algebro-geometric viewpoint (Geometric Invariant Theory, Hilbert schemes, del Pezzo surfaces),
and from combinatorial geometric viewpoint (Matroid Theory).
In the latter approach just linear phenomena are essential,
and in particular,
a generic object of consideration is a {\em simple $n$-configuration}, that is
a set of $n$ points in $\Rp2$ in which no triple of points is collinear.
The dual object is {\em a simple $n$-arrangement}, that is a set of $n$ real lines
containing no concurrent triples.

A combinatorial characterization of a simple $n$-arrangement is its {\it oriented matroid}, which
is roughly speaking a description of the mutual position of its partition polygons. For simple n-configurations it is essentially a
description how do the plane lines
separate the configuration points (see \cite{BLVSWZ} for precise definitions).
Such a combinatorial description was given for simple $n$-arrangements with $n\le7$ in \cite{C} and \cite{W}.
In the beginning of 1980s N.\,Mn\"ev proved his universality theorem and in particular,
constructed examples of combinatorially equivalent
simple configurations which cannot be connected by a deformation.
His initial example with $n\ge19$ was improved by P.\,Suvorov (1988) to $n=14$,
and recently (2013) by Y.Tsukamoto to $n=13$.
Mn\"ev's work motivated the first author to verify in \cite{F1} (see also \cite{F2})
that for $n\le7$ the deformation classification still coincides with the combinatorial one,
or in the other words, to prove connectedness of the realization spaces of the corresponding oriented matroids.
One of applications of this in Low-dimensional topology was found in \cite{KK}, via the link to the
geometry of Campedelli surfaces.

As $n$ grows,
a combinatorial classification of simple $n$-configurations
becomes a task for computer enumeration: there exist 135 combinatorial types of
simple $8$-arrangements (R.\,Canham, E.\,Halsey, 1971, J.\,Goodman and R.\,Pollack, 1980) and 4381 types of simple $9$-arrangements
(J.\,Richter-Gebert, G.\,Gonzales-Springer and G.\,Laffaille, 1989).
The classification includes
analysis of arrangements of {\it pseudolines} (oriented matroids of rank $3$),
their {\it stretchability} (realizability by lines) and
analysis of connectedness of the realization space
of a matroid that gives a deformation classification
(see \cite[Ch.\ 8]{BLVSWZ} for more details).

In what follows, we need only the following summary of
the deformation classification of simple $n$-configurations for $n\le7$.
For $n=5$ it is trivial: simple $5$-configurations form a single deformation component, denoted by
$\LC5{}$.
This is because the points of such a configuration lie on a non-singular conic.
\begin{figure}[h!]
    \centering
    \subfigure{\scalebox{0.43}{\includegraphics{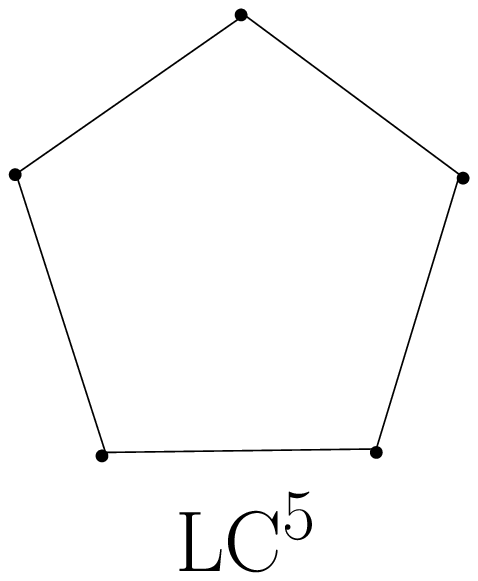}}}\hspace{0.4cm}
    \subfigure[normalsize][cyclic]{\scalebox{0.41}{\includegraphics{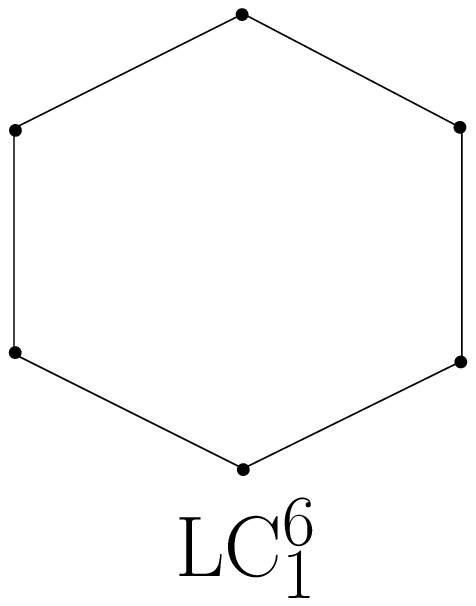}}}\hspace{0.4cm}
    \subfigure[normalsize][bicomponent]{\scalebox{0.41}{\includegraphics{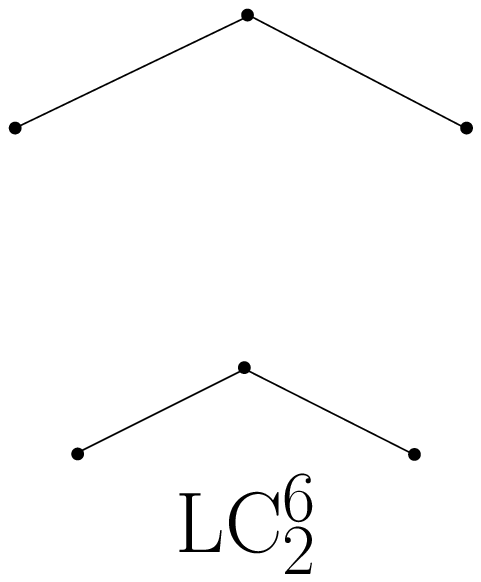}}}\hspace{0.4cm}
    \subfigure[normalsize][tricomponent]{\scalebox{0.41}{\includegraphics{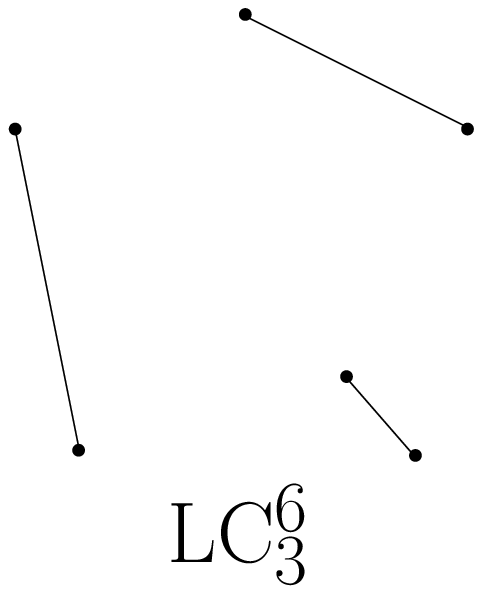}}}\hspace{0.4cm}
	  \subfigure[normalsize][icosahedral]{\scalebox{0.41}{\includegraphics{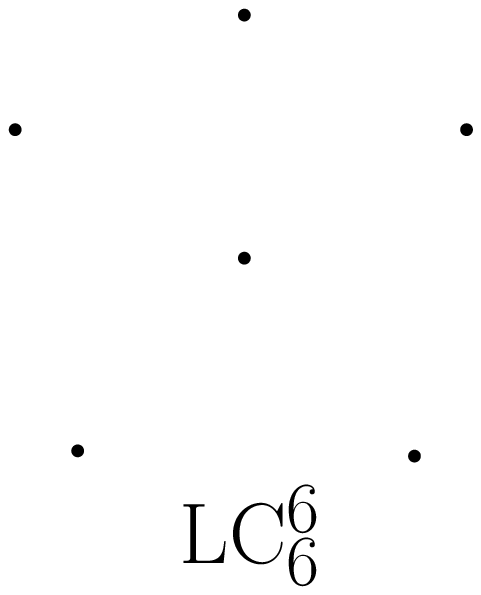}}}
    \caption{Adjacency graphs $\G_\P$ of $5$- and $6$-configurations (cyclic, bicomponent, tricomponent and icosahedral)}
		\label{graphLG56}
  \end{figure}
For $n=6$ there are $4$ deformation classes
shown on
Figure~\ref{graphLG56}.
On this Figure, we sketched configurations $\P$ together with some edges (line segments)
joining pairs of points, $p,q\in\P$. Namely, we sketch
such an edge if and only if it is not crossed by any of the lines connecting
pairs of the remaining $n-2$ points of $\P$.
 The graph, $\G_\P$, that we obtain for a given configuration $\P$
will be called the {\em adjacency graph} of $\P$ (in the context of the oriented matroids, there is
a similar notion of {\it inseparability graphs}).
For $n=6$, the number of its connected components, $1$, $2$, $3$, or $6$,
characterizes $\P$ up to deformation.
The deformation classes of 6-configurations with $i$ components are
denoted $\LC6i$, $i=1,2,3,6$, and the configurations of these four classes are called respectively
{\it cyclic, bicomponent, tricomponent}, and {\it icosahedral} 6-configurations.

Given a simple 7-configuration $\P$, we label a point $p\in\P$ with
an index $\delta=\delta(p)\in\{1,2,3,6\}$
if $\P\sm\{p\}\in \LC6{\delta}$.
Count of the labels gives a quadruple $\s=(\s_1,\s_2,\s_3,\s_6)$, where $\s_k\ge0$ is the number
of points $p\in\P$ with $\delta(p)=k$. We call $\s=\s(\P)$ the {\em derivative code} of $\P$.
 There exist 11 deformation classes of simple 7-configurations
 that are shown on Figure \ref{graphLG7},
together with their adjacency graphs and labels $\delta(p)$.
\begin{figure}[h!]
  \centering
    \subfigure[$\LC{7}{(7,0,0,0)}$]{\scalebox{0.4}{\includegraphics{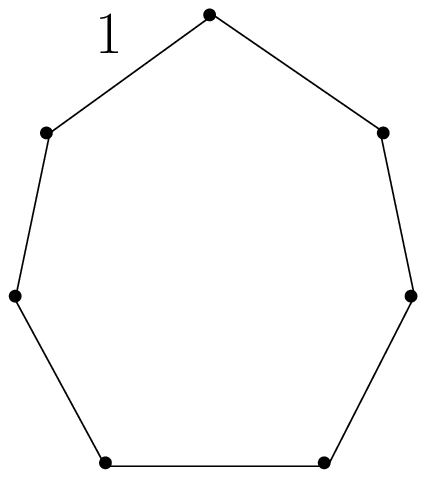}}}\hspace{0.65cm}
    \subfigure[$\LC{7}{(3,4,0,0)}$]{\scalebox{0.4}{\includegraphics{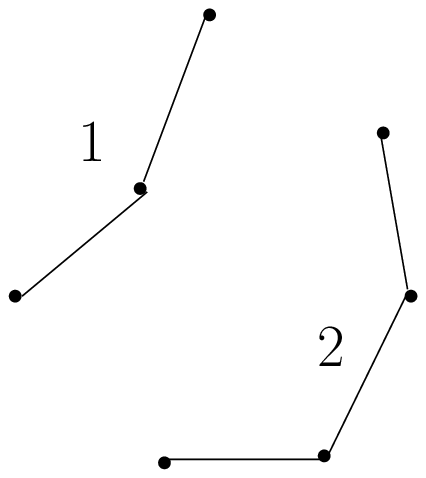}}}\hspace{0.65cm}
    \subfigure[$\LC{7}{(2,2,3,0)}$]{\scalebox{0.4}{\includegraphics{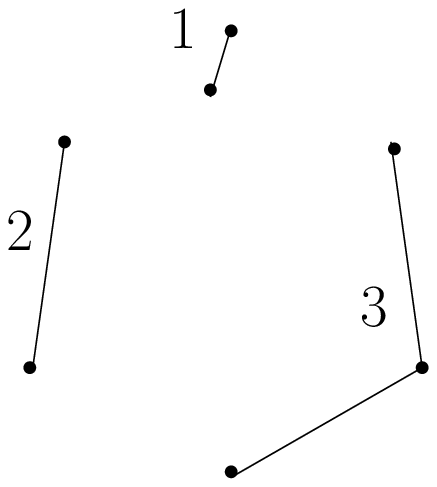}}}\hspace{0.65cm}
    \subfigure[$\LC{7}{(1,2,2,2)}$]{\scalebox{0.4}{\includegraphics{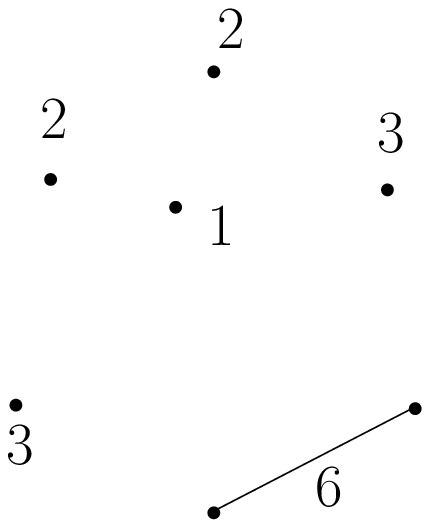}}}\\
		\vspace{0.2cm}
    \subfigure[$\LC{7}{(1,0,6,0)}$]{\scalebox{0.4}{\includegraphics{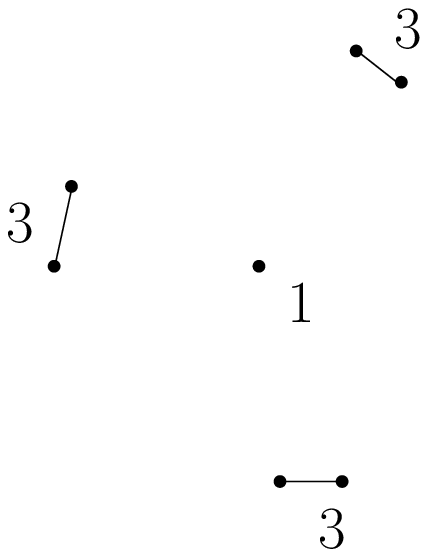}}}\hspace{0.65cm}
    \subfigure[$\LC{7}{(1,6,0,0)}$]{\scalebox{0.4}{\includegraphics{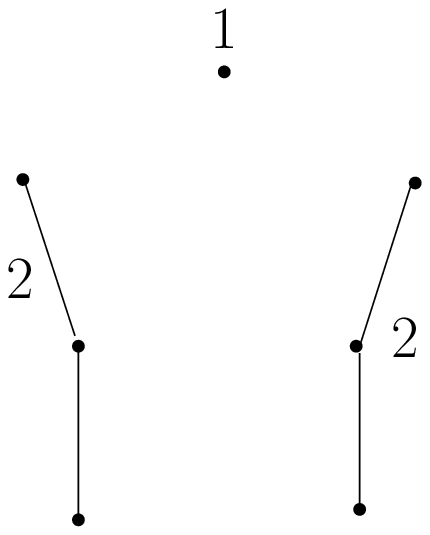}}}\hspace{0.65cm}
    \subfigure[$\LC{7}{(1,4,2,0)}$]{\scalebox{0.4}{\includegraphics{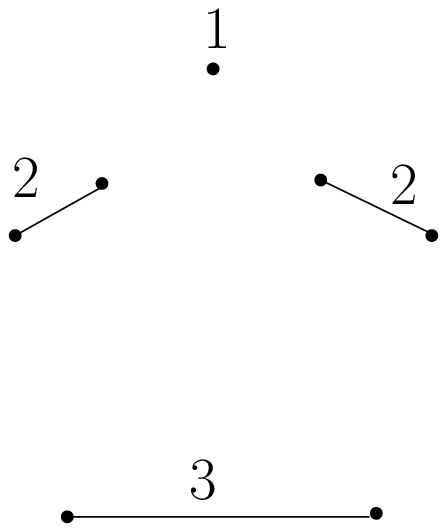}}}\hspace{0.65cm}
		\subfigure[$\LC{7}{(1,2,4,0)}$]{\scalebox{0.4}{\includegraphics{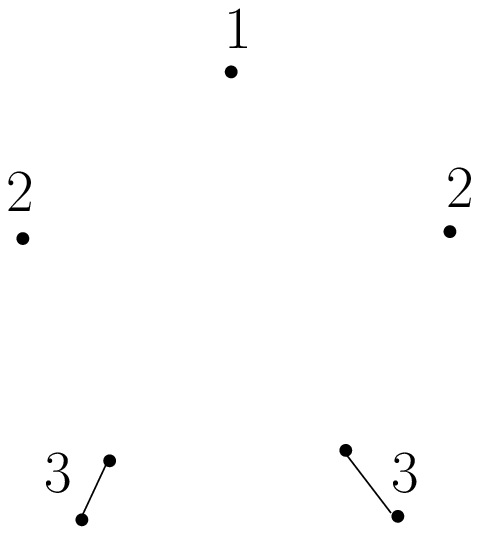}}}\\
		\vspace{0.2cm}
    \subfigure[$\LC{7}{(0,4,3,0)}$]{\scalebox{0.4}{\includegraphics{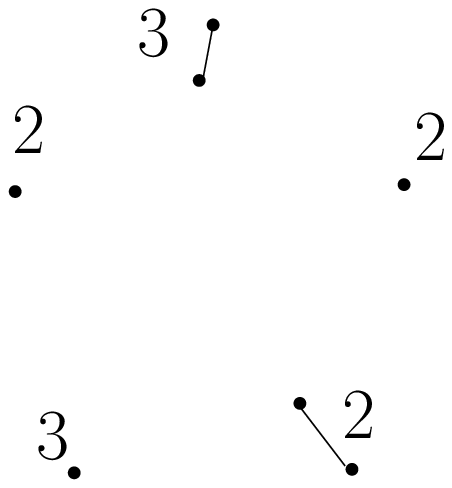}}}\hspace{0.65cm}
 		\subfigure[$\LC{7}{(0,6,1,0)}$]{\scalebox{0.4}{\includegraphics{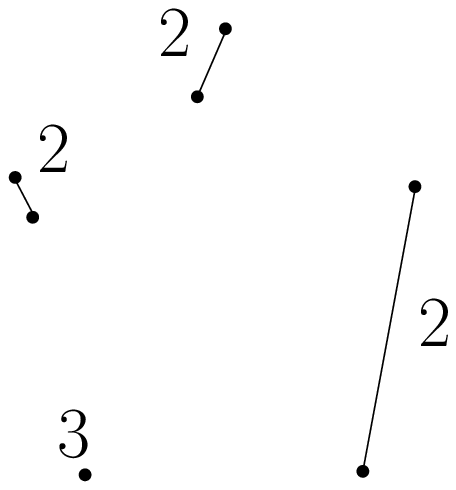}}}\hspace{0.65cm}
		\subfigure[$\LC{7}{(0,3,3,1)}$]{\scalebox{0.35}{\includegraphics{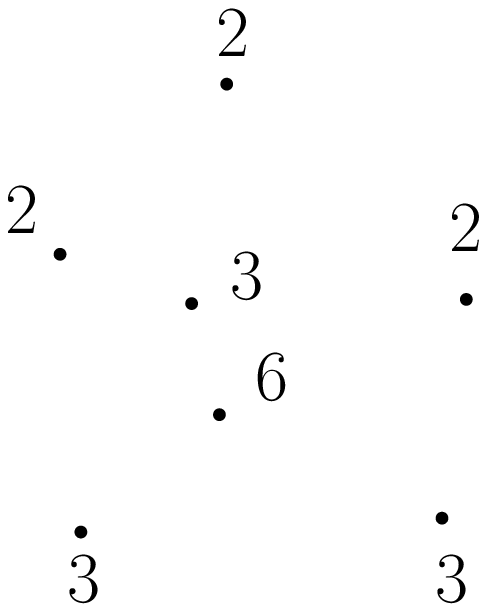}}}
  \caption{Deformation classes of simple 7-configurations}
	\label{graphLG7}
\end{figure}

It is trivial to notice that if $p,q\in\P$ are adjacent vertices in graph $\G_{\P}$,
then $\delta(p)=\delta(q)$,
so, on Figure \ref{graphLG7} we label whole components of $\G_{\P}$ rather than its vertices.
The derivative codes happen to distinguish the deformation classes,
and we denote by $\LC7{\s}$ the class formed by simple $7$-configurations $\P$
with the derivative code $\s$.

\subsection{Typical configurations} Problems related to the linear systems (pencils and nets) of real cubic curves along with related problems
on the real del Pezzo surfaces of degrees 1, 2 and 3 lead to a necessity to refine
the notion  of simple configurations by taking into account also
{\it quadratic degenerations} of configurations, in which six points become {\it coconic} (lying on one conic).
These problems involve also Real Aronhold sets of bitangents to quartics, see Section \ref{concluding}.
 It is noteworthy that a similar motivation (interest to Aronhold sets) was indicated by L.\,Cummings, although in her research
\cite{C} she did not step beyond simple 7-configurations.

The object of our interest is not that well-studied as simple configurations, although
definitely is not new. It appears for example in \cite{GH} in the context of studying Cayley octads
and their relation to the Aronhold sets.
We adopt here the terminology from \cite{GH} and say that
an $n$-configuration is \textit{typical}, if it is simple
and in addition does not contain coconic sextuples of points.
Analyzing the combinatorics of the root system $E_7$ related to
the del Pezzo surfaces associated with typical 7-configurations,
J.~Sekiguchi \cite{Se}
found 14 types of such configurations (these types give some kind of a combinatorial classification).
Later in a joint work with T.\,Fukui (in 1998), he presented a similar computer-assisted enumeration
for typical 8-configurations by analysis of the root system $E_8$.
In a different form,
in terms of separation of configuration points by conics,
a description of typical 7-configurations was given by S.~Le Touz\'e \cite{S3} in the context of studying
the real rational pencils of cubics.
At the same time, a similar combinatorial description of typical 7-configurations was given in
\cite{Z}.
The principal goal in \cite{Z} (and in the current presentation of its results)
is however to give
a more subtle {\it deformation classification} of such configurations, that is to
show that the realization space for each of the 14 combinatorial types of typical $7$-configuration is connected.
 Like deformation classification of simple configurations, this result leads to a deformation classification of certain
 associated real algebro-geometric objects. In the case of typical $7$-configurations such objects are
real del Pezzo surfaces of degree 2 (marked with $7$ exceptional curves), nets of cubics in $\mathbb{RP}^2$, Cayley octads
and nets of quadrics in $\mathbb{RP}^3$. We tried to indicate one of such applications in the end of the paper.

For us, a {\it deformation of $n$-configuration} is simply a path in the corresponding configuration space, or in the other words,
a continuous family $\P_t$, $t\in[0,1]$, formed by $n$-configurations.
We call it {\it L-deformation} if $\P_t$ are simple configurations, and {\it Q-deformation} if $\P_t$
are typical ones.

It is not difficult to observe (see Section 2) that for $6$-configurations the two classifications
coincide: typical 6-configurations can be connected by
an L-deformation if and only if they can be connected by a Q-deformation.
However, for $n>6$, one L-deformation class may contain several Q-deformation classes, and our main goal
is to find their number
in the case of $n=7$, for
each of the 11 L-deformation classes shown on Figure \ref{graphLG7}.

\begin{theorem}\label{14classes} Typical $7$-configurations split into $14$ Q-deformation classes.
Among these classes, two are contained in the class $LC^{7}_{(3,4,0,0)}$, three in the class $LC^{7}_{(2,2,3,0)}$,
and each of the remaining $9$ L-deformation classes of simple $7$-configurations contains just one Q-deformation class.
\end{theorem}
How to subdivide L-deformation classes $LC^{7}_{(3,4,0,0)}$ and $LC^{7}_{(2,2,3,0)}$
into 2 and respectively 3 Q-deformation classes
is shown in Subsections~\ref{non-collabsible} and,
 \ref{decoadj}.

\subsection{Structure of the paper}
In Section~\ref{preliminaries} we recall
the scheme of L-deformation classification from \cite{F1} and
give Q-deformation classification of typical $6$-configurations.
We treat also three cases of $7$-configurations, in which 
connectedness of
the realization spaces is obvious.
Sections 3--5 are devoted to Q-deformation classification for 
the three
existing types of 7-configurations:  heptagonal, hexagonal, and pentagonal.
In the last Section, we discuss some applications including a description of
the 14 real Aronhold sets (Figure~\ref{CAC7graph}). We indicated how this description can
be derived from our results using Cremona transformations. We sketched also an application: a method
(alternative to that of \cite{S1}) to
describe the topology of real rational cubics passing through the points of a 7-configuration.

\subsection{Acknowledgments}
This paper is essentially based on \cite{Z}, which was partially motivated by
our attempts to understand and develop the results of \cite{S1}.
%
\section{Preliminaries}\label{preliminaries}
\subsection{The monodromy group of a configuration}\label{monodromy}
By the {\it L-deformation monodromy group} of a simple $n$-configuration $\P$ we mean the subgroup, $\Aut_L(\P)$,
of the permutation group $S(\P)$ realized by L-deformations, that is
 the image in $S(\P)$ of the fundamental group of the L-deformation component of $\P$
(using some fixed numeration of points of $\P$, we can and will
identify $S(\P)$ with the symmetric group $S_n$).
For a typical $n$-configuration $\P$, we similarly define the {\it Q-deformation monodromy group}
$\Aut_Q(\P)\subset S(\P)\cong S_n$ formed by the permutations realized by Q-deformations.

In the case $n=4$, any permutation can be realized by a deformation (and in fact, by a projective
transformation), so we have $\Aut_L(\P)=\Aut_Q(\P)=S_4$.
For $n=5$, we obtain the dihedral group $\Aut_L(\P)=\Aut_Q(\P)=\DD_5$ associated to the pentagon $\G_\P$ (as it was noted
the adjacency graph is an L-deformation invariant).

More generally, we can consider a class of simple {\it $n$-gonal $n$-configurations}, $\P$,
that are defined as ones forming a convex $n$-gon in the complement $\R^2=\Rp2\sm \ell$, of some line $\ell\subset\Rp2\sm\P$.
For $n\ge5$ this $n$-gon (that coincides with the adjacency graph $\G_\P$)
is preserved by the monodromy group action, and it is easy to conclude that $\Aut_L(\P)=\DD_n$.
In particular, $\Aut_L(\P)=\DD_6$ for $\P\in\LC61$.

\begin{remark}
It is also not difficult to show (see \cite{F1}) that for $6$-configurations $\P$ from components
$\LC62$, $\LC63$, and $\LC66$, groups $\Aut_L(\P)$ are
respectively $\Z/4$, $\DD_3$, and the icosahedral group. These facts however are not used in our paper.
\end{remark}

\subsection{$\Aut(\P)$-action on L-polygons}\label{LGC6}
The $\binom{n}2$ lines passing through the pairs of points of a simple $n$-configuration $\P$
divide $\Rp2$ into polygons that we call \textit{L-polygons associated to $\mathcal{P}$}. Group $\Aut_L(\P)$ acts naturally on the set of those 
 L-polygons that cannot be collapsed
in a process of L-deformation.

It is easy to check that 
 for $n=5$ none of the 31 L-polygons can be collapsed. Thus,
we obtain an action of $\Aut_L(\P)=\DD_5$, which divide the set of $31$ L-polygons into 6 orbits:
three {\it internal orbits} formed by L-polygons lying inside pentagon $\G_\P$ and three
{\it external} ones, placed outside $\G_\P$ (see Figure \ref{4lpolygons}).
\begin{figure}[h]
\centering
\scalebox{0.3}{\includegraphics{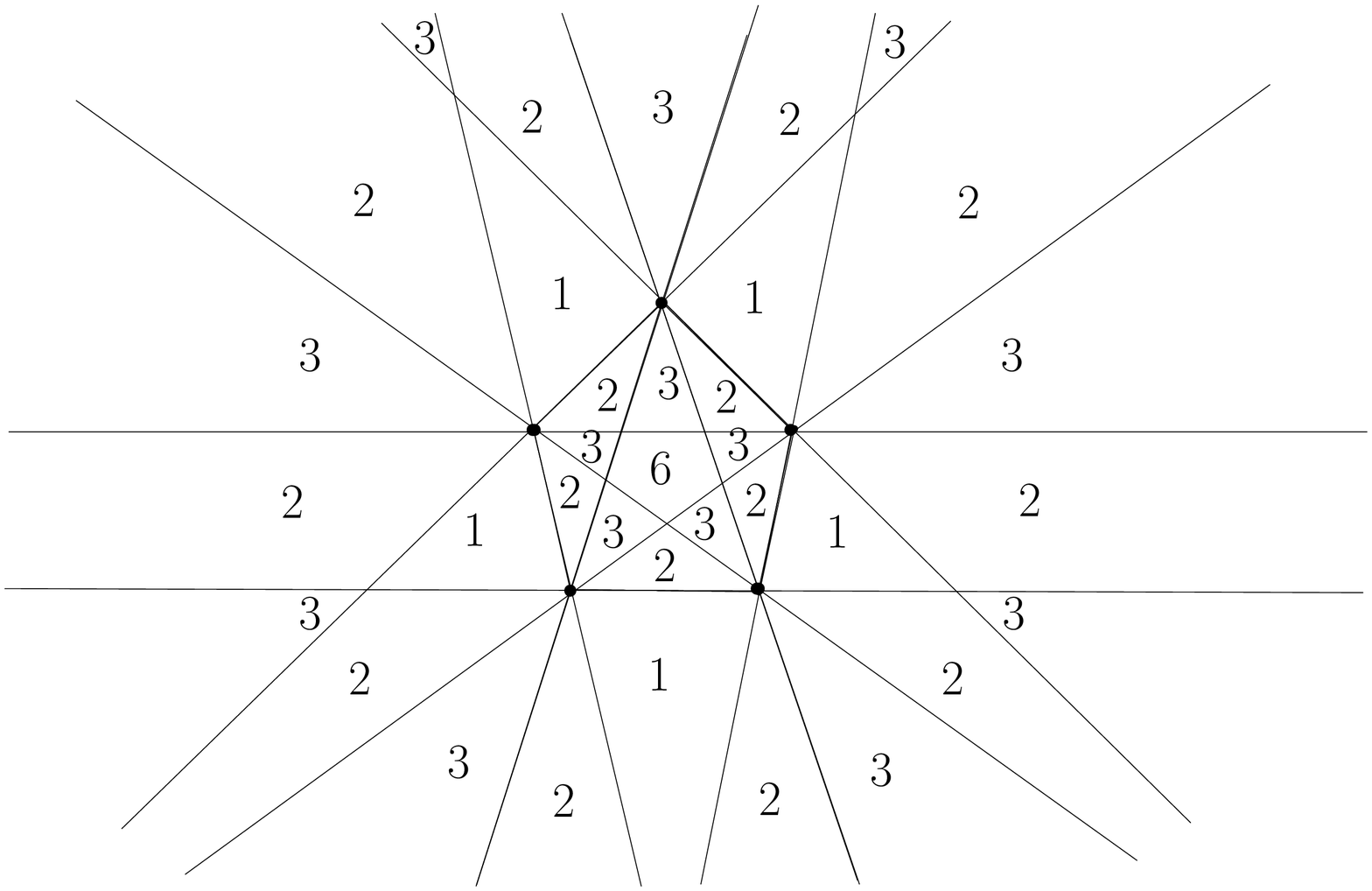}}
\caption{$\DD_5$-orbits of L-polygons for $\P\in \LC5{}$.
Labels $i=1,2,3,6$ represent the corresponding classes $\LC6{i}\ni\P'$.
Three internal orbits (inside the pentagon) are labeled by $2$, $3$, $6$ and three
external orbits (outside the pentagon) by $1$, $2$, $3$}
\label{4lpolygons}
\end{figure}

By adding a point $p$ in one of the L-polygons associated to $\P$ we obtain
a simple $6$-configuration, $\P'=\P\cup\{p\}$.
The L-deformation class of $\P'$ depends obviously only on the $\Aut_L(\P)$-orbit of the L-polygon containing $p$.
Figure \ref{4lpolygons} shows the correspondence between the six orbits and the four classes $\LC6{i}$.
Note that for $i$ equal to $1$ and $6$
class $\LC6{i}$ is represented by one orbit, while for $2$ and $3$ such class is represented by two orbits.
This is because $\Aut_L(\P)$ acts transitively on the points of $\P$ for $i=1,6$, while for $i=2,3$
the vertices of $\P$ split into two $\Aut_L(\P)$-orbits.

\subsection{The dual viewpoint}\label{dual}
In the dual projective plane $\widehat{\Rp2}$, consider the arrangement of lines
$\widehat{\P}=\{\widehat p_1,\dots,\widehat p_n\}$ which is
dual to a given $n$-configuration $\P=\{p_1,\dots,p_n\}$.
Lines $\widehat p_i$ divide $\widehat{\Rp2}$ into polygons that we call
{\it subdivision polygons} of $\widehat{\P}$.
We define the {\it polygonal spectrum} of an $n$-configuration $\P$ as the $(n-2)$-tuple
$f=(f_3,f_4,\dots,f_n)$, where $f_k$ is the number of $k$-gonal subdivision polygons of $\widehat{\P}$.
Euler's formula easily implies that $\sum_{k=3}^n(k-4)f_k=-4$ for simple $n$-configurations.
It is easy to see also that $f_3\ge5$ if $n\ge5$ (in fact, it is known that $f_3\ge n$), which
implies that for $n\ge5$, at least one subdivision polygon has $5$ or more sides.

If $\ell\subset\Rp2 \sm \P$ is a line, then
point $\widehat\ell\in\widehat{\Rp2}$ that is dual to $\ell$
should belong to one of the
subdivision polygons, say $F$. Then $F$ is an $m$-gon if and only if
the convex hull, $H$, of $\P$ in the affine plane $\Rp2\sm\ell$ is an $m$-gon.
Note that
$m$-gons $F$ and $H$ are dual:
points $p\in H$ are dual to lines $\widehat p$ disjoint from $F$ (and vice versa).

This gives two options for a simple $6$-configuration $\P$.
The first option is $f_6>0$ that implies $\P\in\LC61$.
The second option is $f_6=0$, $f_5>0$, which
means that the convex hull of $\P$ is a pentagon in some affine chart $\Rp2\sm\ell$.

A simple $7$-configurations is called {\it heptagonal} if $f_7>0$,
{\it hexagonal} if $f_7=0$ and $f_6>0$, and {\it pentagonal} if $f_6=f_7=0$ and $f_5>0$.
Note that $f_5+f_6+f_7>0$ for any simple 7-configuration, and so, one of these three conditions is satisfied.
In terms of the affine chart $\R^2=\Rp2\sm\ell$, these three cases give (if point $\hat\ell$ is chosen inside a subdivision polygon with the maximal
number of sides):
$7$ points forming a convex heptagon, $6$ points forming a convex hexagon plus a point inside it,
and $5$ points forming a convex pentagon plus two points inside it (see Table below).
\begin{table}[h!]
\caption{Derivative codes and polygonal spectra}
\centering
\label{spectrandcodes}
\scalebox{0.9}
{\begin{tabular}{|l c c|}
\hline%
$\mathcal{P}\in \LC7{}$ &
$\s=(\s_{1},\s_{2},\s_{3},\s_{6})$ & $f=(f_{3},f_{4},f_{5},f_{6},f_{7})$ \\
\hline
Heptagonal & $(7,0,0,0)$ & $(7,14,0,0,1)$ \\
\hline
Hexagonal & $(3,4,0,0)$ & $(7,13,1,1,0)$ \\
					& $(2,2,3,0)$ & $(8,11,2,1,0)$  \\
					& $(1,2,2,2)$ & $(11,5,5,1,0)$  \\
					& $(1,0,6,0)$ & $(9,9,3,1,0)$   \\ \cline{2-3}
\hline
Pentagonal with $\sigma_{1}=1$ & $(1,6,0,0)$ & $(7,12,3,0,0)$ \\
													& $(1,4,2,0)$ & $(8,10,4,0,0)$  \\
													& $(1,2,4,0)$ & $(9,8,5,0,0)$   \\ \cline{2-3}
Pentagonal with $\sigma_{1}=0$ & $(0,4,3,0)$ & $(8,10,4,0,0)$  \\
													& $(0,6,1,0)$ & $(7,12,3,0,0)$  \\
													& $(0,3,3,1)$ & $(10,6,6,0,0)$  \\ \cline{2-3}
\hline
\end{tabular}}
\end{table}

\subsection{Simple $7$-configurations with $\s_1>0$}\label{simple7}
Following \cite{F1}, we outline here the L-deformation classification in the most essential for us case
of simple $7$-configurations with $\s_1>0$. By definition, such configurations can be
presented as $\P'=\P\cup\{p\}$, where $\P$ is
its cyclic 6-subconfiguration and $p$ is an additional point (there are precisely $\s_1$ ways to choose such
a decomposition of $\P'$).

Like in the case of 6-configurations considered above, the monodromy group
$\Aut_L(\P)=\DD_6$ of $\P\in\LC61$ acts on the L-polygons associated to $\P$
(see Figure \ref{8lregions}), and the
L-deformation class of $\P'$ depends only on the orbit ($[A],[B],\dots$)
of the L-polygon that contains point $p$.
\begin{figure}[h]
\begin{center}
{\scalebox{0.38}{\includegraphics{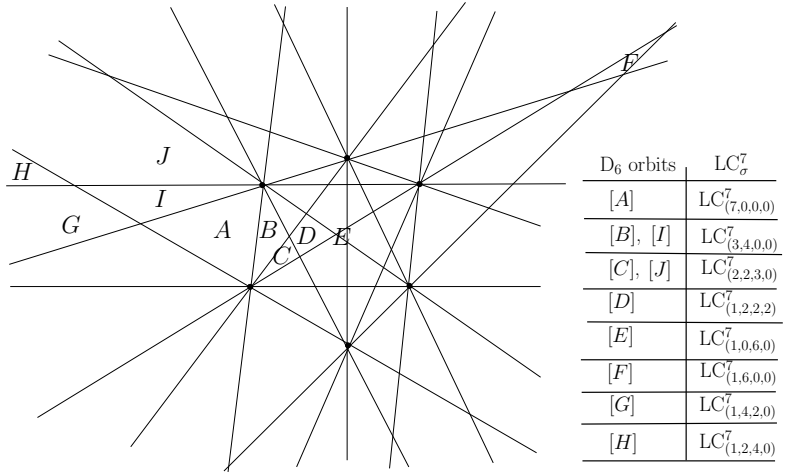}}}
\end{center}
\caption
{Four internal L-polygons, $B,C,D,E$, and six external L-polygons, $A,F,G,H,I,J$, that represent ten
$\DD_6$-orbits}
\label{8lregions}
\end{figure}
An additional attention is required to the four L-polygons that can be collapsed, namely, the "central"
triangle $E$ in Figure \ref{8lregions} and three other triangles, one of which is marked as $F$ on this Figure:
it is bounded by a principal diagonal and the two sides of the hexagon
that have no common vertices with this diagonal.
We skip here the arguments from \cite{F1} (but after passing to the Q-deformation
classification in Propositions \ref{specialqdeformation} and \ref{F-connect},
we will provide in fact a more subtle version of this proof).

Note that $\DD_6$-action is well-defined on the contractible
L-polygons as well as on the non-contactable ones: this action preserves
$E$ invariant and naturally permutes the three polygons of type $F$ so that they
form a single orbit, $[F]$.

It is easy to see that $f_6+f_7\le1$ for any simple 7-configuration, and that our assumption on
$\P'$ admits three options. The first option is $f_7=1$, that is to say, $\P'$ is a heptagonal configuration.
This correspond to location of $p$ inside one of the six L-polygons from the $\DD_6$-orbit $[A]$ of a triangle $A$.
The second option is location of $p$ inside hexagon $\G_\P$, in one of the {\it internal L-polygons}
from the orbits $[B]$, $[C]$, $[D]$, or $[E]$ (see Figure \ref{8lregions}). In this case, $f_7=0$ and
$f_6(\P')=1$, that is to say, configuration $\P'$ is hexagonal.
In the remaining case, $p$ lies outside $\G_\P$, but not inside one of the six triangles of type $A$.
Then $\P'$ may be either hexagonal with $\s_1\ge2$, or pentagonal with $\s_1\ge1$.

 Totally, we enumerated 8 L-deformation classes out of 11. The remaining 3 pentagonal classes with $\s_1=0$
can be described by placing two points inside a convex affine 5-configuration, as it can be understood from
Figure \ref{graphLG7} (for details see \cite{F1}).

\subsection{Coloring of graphs $\G_\P$ for typical $6$-configurations}\label{QGC}
Given a typical $6$-configuration $\P$, we say that its point $p\in\P$ is \textit{dominant} (\textit{subdominant})
if it lies outside of (respectively, inside) conic
$Q_p$ that passes through the remaining $5$ points of $\P$.
Here, by points {\it inside} ({\it outside of}) $Q_p$ we mean points lying in the component
of $\Rp2\sm Q_p$ homeomorphic to a disc, (respectively, in the other component).
We color the vertices of adjacency graph $\G_\P$:
the dominant points of $\P$ in black
and subdominant ones in white, see Figure \ref{adjgraph6} for the result.
\begin{figure}[h]
 \centering
	\subfigure[$\QC61$]{\scalebox{0.37}{\includegraphics{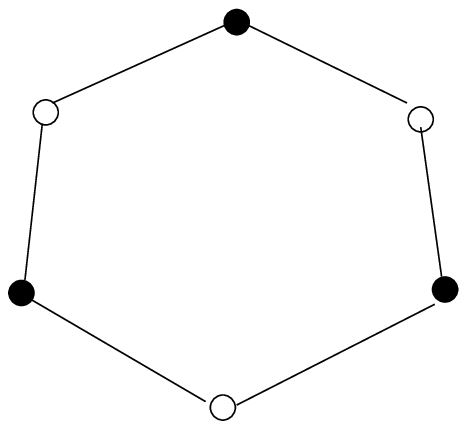}}}\;\;\;\;\;
  	\subfigure[$\QC62$]{\scalebox{0.37}{\includegraphics{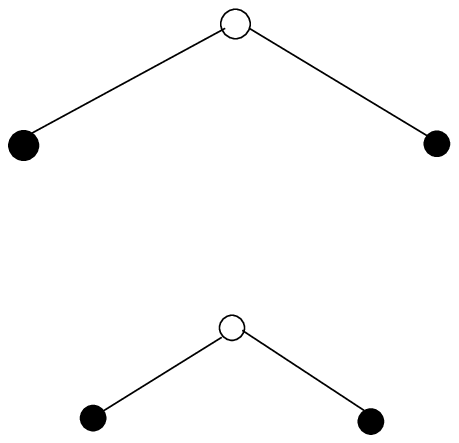}}} \;\;\;\;\;
    \subfigure[$\QC63$]{\scalebox{0.37}{\includegraphics{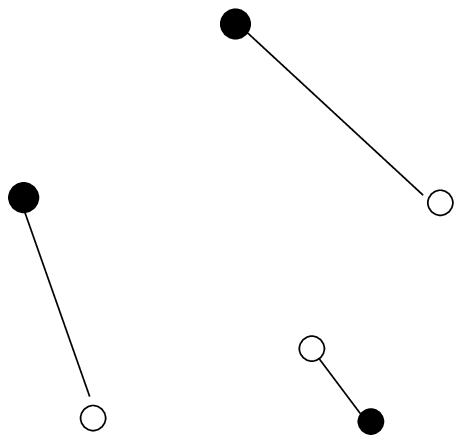}}}\;\;\;\;\;
    \subfigure[$\QC66$]{\scalebox{0.37}{\includegraphics{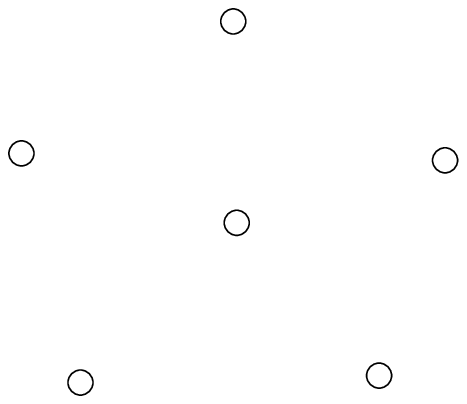}}}\\
    \caption[The four $Q$-deformation classes in $QC^6$]{Decorated adjacency graphs of typical $6$-configurations}
  \label{adjgraph6}
\end{figure}

Graphs $\G_\P$ are bipartite, \ie,  adjacent vertices have different colors.

\begin{lemma} \label{sd} For a typical $6$-configuration $\P$,
every edge of $\G_\P$
connects a dominant and a subdominant points.
\end{lemma}
\begin{proof} It follows from analysis of the pencil of conics passing through $4$ points of $\P$:
 a singular conic from this pencil cannot intersect an edge of $\G_\P$ connecting the remaining two points.
\end{proof}

\subsection{$Q$-deformation components of typical $6$-configurations}\label{6confs}
Let $\QD{n}\subset\LC{n}{}$ denote the subset formed by simple $n$-configurations, which have a subconfiguration
of six points lying on a conic. Then $\QC{n}{}=\LC{n}{}\sm\QD{n}$ is the set of typical $n$-configurations.

Note that $\QD6\subset\LC61$, so, $\LC6i$ for $i=2,3,6$ are formed entirely by typical configurations and thus,
give three Q-deformation components $\QC6i=\LC6i$, $i=2,3,6$, of typical $6$-configurations.
It follows immediately also that $\Aut_Q(\P)=\Aut_L(\P)$ for $\P$ from these three Q-deformation components.

To complete the classification it is left to observe connectedness of $\QC61=\LC61\sm\QD6$, which implies that $\QC61$ is the remaining Q-deformation component in $\QC6{}$.

Connectedness follows immediately from the next Lemma and the fact that any $6$-configuration $\P\in \QC{6}{1}$ has a dominant point (in fact, it has exactly three such points, see Figure \ref{adjgraph6}).

\begin{lemma} \label{quadraticsix} Consider two hexagonal $6$-configurations $\P^i\in\QC61$, with marked dominant points
$p^i\in\P^i$, $i=0,1$.
Then, there is a $Q$-deformation $\P^t$, $t\in[0,1]$ that takes $p^0$ to $p^1$.
\end{lemma}

\begin{proof}
The same idea as in Subsection \ref{LGC6} is applied: the triangular L-polygons marked by $1$ on Figure \ref{4lpolygons}
 are divided into pairs of {\it Q-regions}
 by the conic passing through the vertices of a pentagon.
A sixth point placed outside (inside) of the conic is dominant (subdominant).
The monodromy group $\Aut_L(\P^i\sm\{p^i\})=\Aut_Q(\P^i\sm\{p^i\})=\DD_5$, $i=0,1$, 
acts transitively on the Q-regions of the same kind (in our case, on the parts of triangles marked by 1 that lie outside the conic), and these regions cannot be
 contracted in the process of L-deformation of the pentagon. Therefore, an L-deformation between
 $\P^0\sm\{p^0\}$ and $\P^1\sm\{p^1\}$ that brings the Q-region containing $p^0$ into the one containing $p^1$ can be extended to a required Q-deformation $\P^t$ (see Figure~\ref{regionfor6pts}).
\begin{figure}[h]
  \centering
  {\scalebox{0.28}{\includegraphics{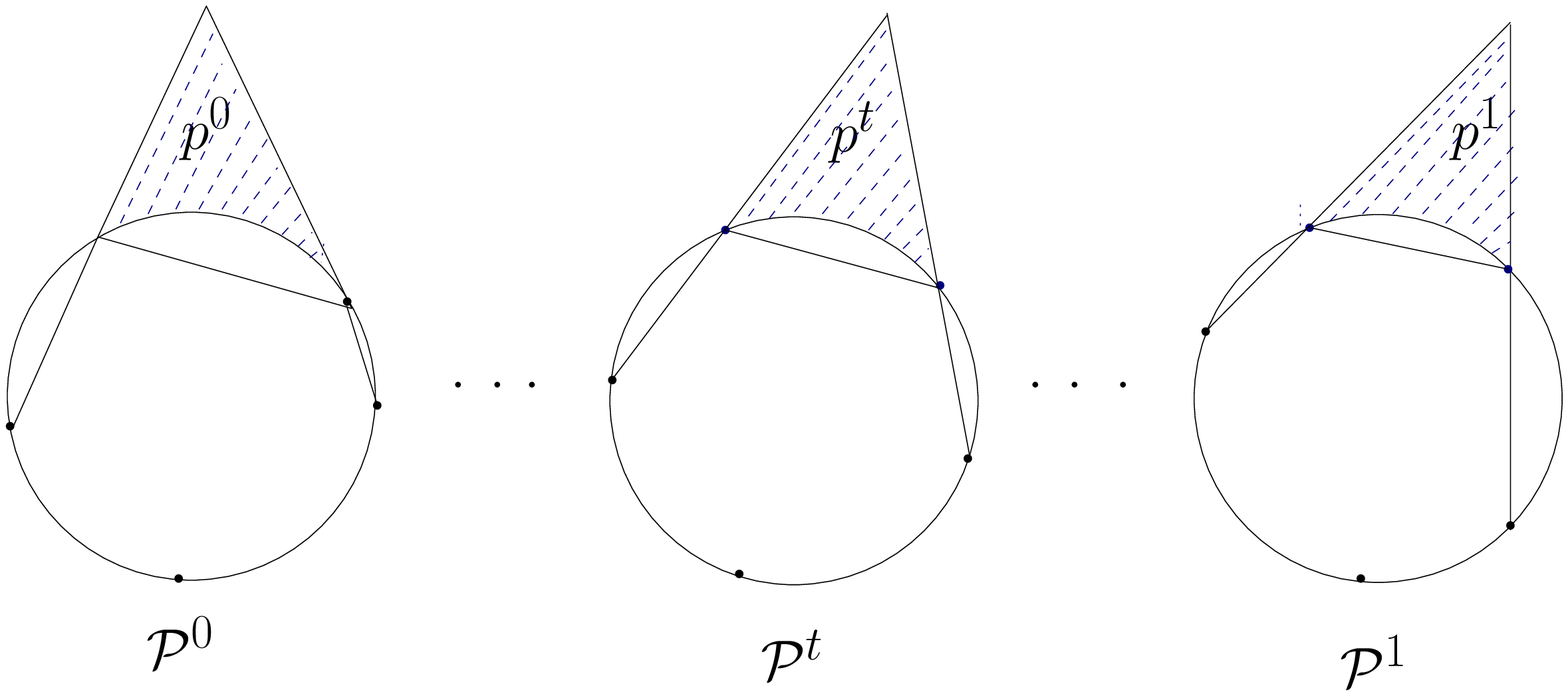}}}
  \caption{Continuous family of triangles of type ``$1$''}
	\label{regionfor6pts}
\end{figure}
\end{proof}
It follows easily that $\Aut_Q(\P)\cong\DD_3$, for $\P\in\QC61$, namely, $\DD_3$
 is a subgroup of $\Aut_L(\P)\cong\DD_6$
that preserves the colors of vertices of graph $\G_\P$ on Figure \ref{adjgraph6}.

\subsection{Q-deformation components of typical $7$-configurations: trivial cases}\label{3trivial}
Let $\QC7{\s}=\LC7{\s}\sm\QD7$, where $\s$ is one of the 11 derivative codes of $7$-configurations
(see Figure \ref{graphLG7} or Table \ref{spectrandcodes}).
Like in the case of 6-configurations,
some of the L-deformation components $\LC7{\s}$, namely, the ones with $\s_1=0$
are disjoint from $\QD7$, therefore in these cases $\QC7{\s}=\LC7{\s}$ are Q-deformation components.
From Table \ref{spectrandcodes}, this holds for $\s$ being $(0,4,3,0)$, $(0,6,1,0)$, and
$(0,3,3,1)$.

\section{Heptagonal $7$-configurations}\label{hep1}

\subsection{Dominance indices}
We shall prove in Subsection \ref{heptagonal-Qdeformations} connectedness of space $\QC7{(7,0,0,0)}$ formed by heptagonal typical configurations.
For a fixed configuration $\P\in\QC7{(7,0,0,0)}$ and any pair of points $p,q\in\P$ let us denote by $Q_{p,q}$
the conic passing through the other five points of $\P$.
For each $p\in\P$ let us denote by $d(p)$ the number of points $q\in \P\sm\{p\}$
for which $p$ lies outside of the conic $Q_{p,q}$; this number $0\le d(p)\le 6$ will be called
the \textit{dominance index} of $p$.

The crucial fact for proving connectedness of $\QC7{(7,0,0,0)}$
is existence of a point $p\in\P$ such that $d(p)=6$.
%
%
We shall prove more: among the 14 ways to numerate cyclically the vertices of heptagon $\G_\P$ (starting from any vertex, one can go around
in two possible directions), one can distinguish a particular one that we call the {\it canonical cyclic numeration}.

\begin{proposition}\label{cyclicorder}
For any $\P\in\QC{7}{(7,0,0,0)}$, there exists a canonical cyclic numeration of its points, $\P=\{p_0,\dots,p_6\}$, such that
$d(p_k)$ is $k$ for odd $k$ and $6-k$ for even. In the other words, the sequence of $d(k)$ is $6,1,4,3,2,5,0$.
\end{proposition}

One can derive this proposition from the results of \cite[Sec.\,2.1]{S2}, but we give below a proof based on different (in our opinion, more transparent)
approach.
The first step of our proof is the following observation.

\begin{lemma}\label{unique}
For any $\P\in\QC{7}{(7,0,0,0)}$, there exists at most one point $p\in\P$ with the dominance index $d(p)=6$ and at most one with
$d(p)=0$.
\end{lemma}
\begin{proof} Assume that by contrary, $d(p)=d(q)=6$ for $p,q\in\P$. Then $p$ and $q$ are dominant points in
$6$-configuration $\P_r=\P\sm\{r\}$ for any $r\in\P\sm\{p,q\}$. But dominant and subdominant points in hexagon
 $\G_{\P_r}$ are alternating (see Figure \ref{adjgraph6}), and so,
 the parity of the orders of
dominant points, with respect to a cyclic numeration of the hexagon vertices, is the same.
By an appropriate choice of $r$, this parity however can be made different, which lead to a contradiction.
In the case $d(p)=d(q)=0$ a proof is similar.
\end{proof}

\subsection{Position of the vertices and edges of $\G_\P$ with respect to conics $Q_{i,j}$}
Let us fix any cyclic numeration $p_0,\dots,p_6$ of points of $\P\in\QC{7}{(7,0,0,0)}$. We
denote by $Q_{i,j}$ the conic passing through the points of $\P$ different from $p_i$ and $p_j$ and put
 $\idx{i}{j}=0$ if $p_i$ lies inside conic $Q_{i,j}$ and $\idx{i}{j}=1$ if outside, $0\le i,j\le6$, $i\ne j$.
By definition, we have
\begin{eqnarray}
\label{sumidx}
d(p_{i})=\sum_{0\leq j\leq 6, j\ne i} \idx{i}{j}, \quad i=0,\dots,6.
\end{eqnarray}

In what follows we apply ``modulo 7'' index convention in notation for $p_i$, $Q_{i,j}$ and $d_{i,j}$,
that is put $p_{i+1}=p_0$ if $i=6$, $p_{i-1}=p_6$ if $i=0$, etc.

\begin{lemma}\label{indexcalculation} Assume that $0\le i\le6$. Then
\begin{enumerate}
\item [(a)]
$\idx{i}{j}+\idx{i+1}{j}=1$ for all $0\le j\le 6$, $j\ne i,i+1$.
\item [(b)] $\idx{i}{i+1}=\idx{i+1}{i}$ and $\idx{i-1}{i}=\idx{i}{i-1}$ provided $d(p_i)\ne0,6$.
\item [(c)] $\idx{i-1}{i}\ne\idx{i}{i+1}$ provided  $d(p_i)\ne0,6$.
\end{enumerate}
\end{lemma}
\begin{proof}
(a) follows from Lemma~\ref{sd} applied to $\P\sm\{p_j\}$.
Assume that (b) does not hold, say, $\idx{i}{i+1}=1$ and $\idx{i+1}{i}=0$ (the other case is analogous).
This means that $p_i$ lies outside of conic $Q_{i,{i+1}}$ and $p_{i+1}$ lies inside.
Since $d(p_i)\ne6$, there is another conic, $Q_{i,j}$ containing $p_i$ inside. This contradicts to the
Bezout theorem, since $Q_{i,{i+1}}$ and $Q_{i,j}$ have $4$ common points $p_k$, $0\le k\le6$, $k\ne i,i+1,j$,
and in addition one more point as it is shown on Figure~\ref{examplemutual}.
\begin{figure}[h!]
\centering
\scalebox{0.35}{\includegraphics{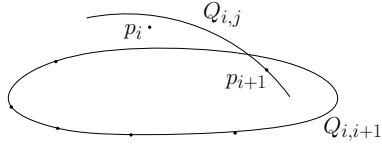}}
\caption{An extra intersection point of conics $Q_{{i},{i+1}}$ and $Q_{i,j}$}
\label{examplemutual}
\end{figure}
For proving (c) we apply Lemma ~\ref{sd} to the cyclic 6-configuration $\P\sm\{p_i\}$,
in which points $p_{i-1}$ and $p_{i+1}$ become consecutive, and thus, one and only one of them is dominant, say
$p_{i-1}$ (the other case is analogous). Then $\idx{i-1}{i}=1$ and $\idx{i+1}{i}=0$, and thus, $\idx{i}{i+1}=0$
as it follows from (b).
\end{proof}

We say that an edge $[p_i,p_{i+1}]$ of heptagon $\G_\P$ is {\it internal} (respectively, {\it external})
if its both endpoints lie inside (respectively, outside of)
conic $Q_{i,i+1}$, or in the other words, if $\idx{i}{i+1}=\idx{i+1}{i}=0$ (respectively, if
$\idx{i}{i+1}=\idx{i+1}{i}=1$).
If one endpoint lies inside and the other outside, we say that
this edge is {\it special} (see Figure \ref{edecoration}a--c).
\begin{figure}[h!]
\centering
{\scalebox{0.35}{\includegraphics{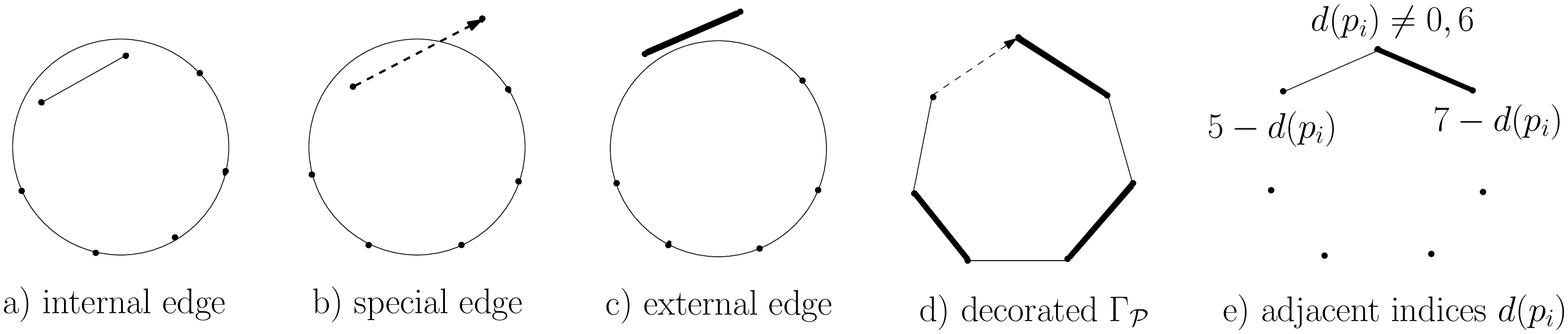}}}
\caption{Decoration of edges. Indices $d(p_i)$ for adjacent vertices.}
\label{edecoration}
\end{figure}
\begin{cor}\label{extremal} A special edge of the heptagon $\G_\P$ should connect a vertex of
dominance index $0$ with a vertex of index $6$, and in particular, such an edge is unique if exists.
The internal and external edges are consecutively alternating. In particular, a special edge must exist (since the number of edges is odd).
\qed
\end{cor}

We sketch the internal and external edges of $\Gamma(\mathcal{P})$ respectively as thin and thick ones.
The special edge is shown dotted and directed from the vertex of dominance index $0$ to the one of dominance index $6$.
Corollary \ref{extremal} means that graph $\Gamma(\mathcal{P})$ decorated this way should look like is shown on
Figure \ref{edecoration}d.

\begin{lemma}\label{5or7} The sum $d(p_i)+d(p_{i+1})$ is $5$ if edge $[p_i,p_{i+1}]$ is internal, and is $7$ if external.
\end{lemma}
\begin{proof}
By (\ref{sumidx}), we have
\begin{eqnarray*} d(p_{i})+d(p_{i+1})
&=& (\idx{i}{i+1}+\idx{i+1}{i})+\sum_{\substack{
   0\leq j \leq 6 \\
   j\neq i,i+1
  }}(\idx{i}{j}+\idx{i+1}{j}),
\end{eqnarray*}
where by Lemma \ref{indexcalculation}(b), $(\idx{i}{i+1}+\idx{i+1}{i})=2\idx{i}{i+1}$ is $0$ if edge $[p_i,p_{i+1}]$
is internal and $2$ if external. By Lemma \ref{indexcalculation}(a), the remaining sum is $5$.
\end{proof}

\subsection{Proof of Proposition~\ref{cyclicorder}}
Lemma \ref{5or7} together with Corollary \ref{extremal} let us recover the whole sequence $d(p_i)$, $i=0,\dots6$, from any value
different from $0$ and $6$ (which exists by Lemma \ref{unique}), see Figure \ref{edecoration}e.

\subsection{Connectedness of $\QC{7}{(7,0,0,0)}$}\label{heptagonal-Qdeformations}
This proof is similar to the proof of connectedness of $\QC61$ in Subsection \ref{6confs}.
 Given $\P\in \QC{7}{(7,0,0,0)}$, assume that its points $p_0,\dots,p_6$ have the canonical cyclic
numeration, and consider subconfiguration $\P_0=\P\sm\{p_0\}\in\QC61$. As we observed in Subsection \ref{simple7},
point $p_0$ lies in a triangular L-polygon $A$ associated with $\P_0$ (see Figure \ref{8lregions}).

Such a triangle $A$ is subdivided into 6 or 7 {\it Q-regions} $A_i$ by the conics $Q_i=Q_{0,i}$, $i=1,\dots,6$,
that connect quintuples of points of $\P_0$, see Figure \ref{adjgraph6}.
\begin{figure}[h!]
  \centering
    \subfigure[Six $Q$-regions]{\scalebox{0.4}{\includegraphics{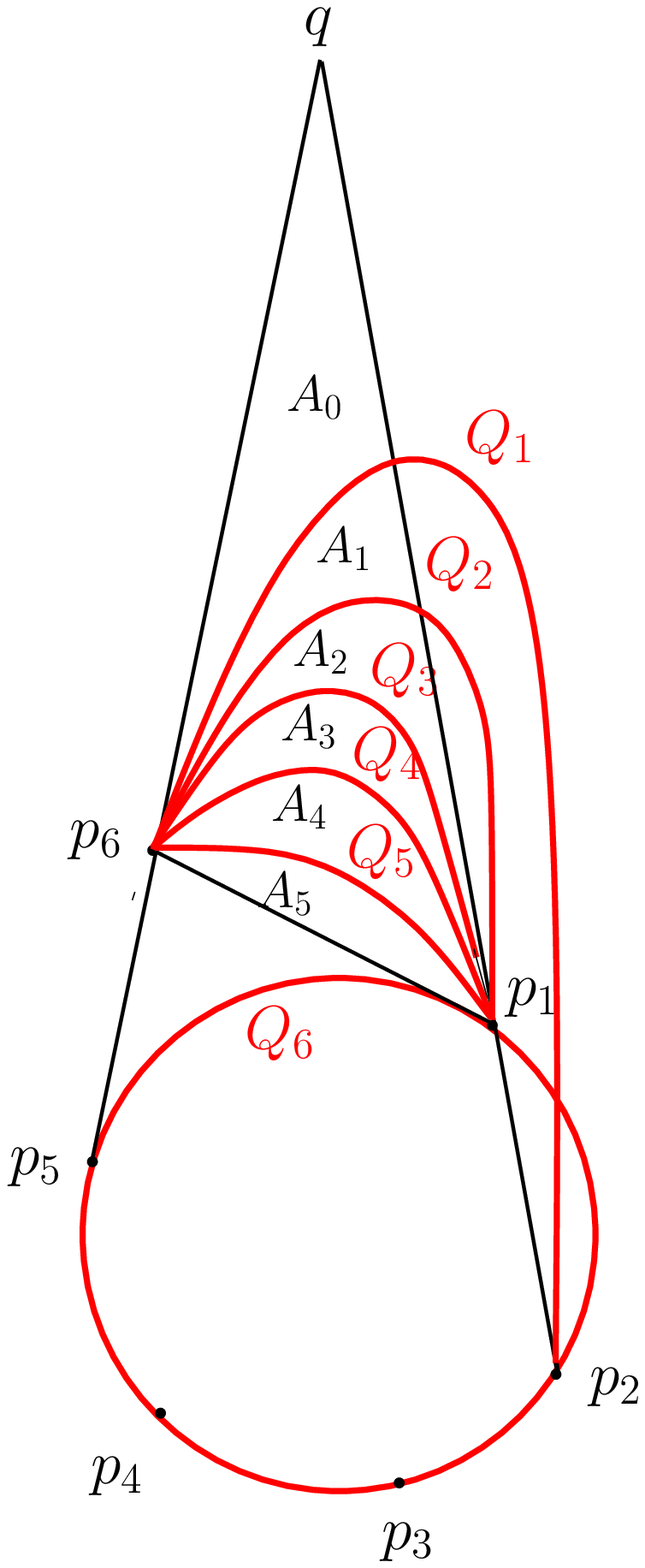}}}\hspace{2cm}
  	\subfigure[Seven $Q$-regions]{\scalebox{0.4}{\includegraphics{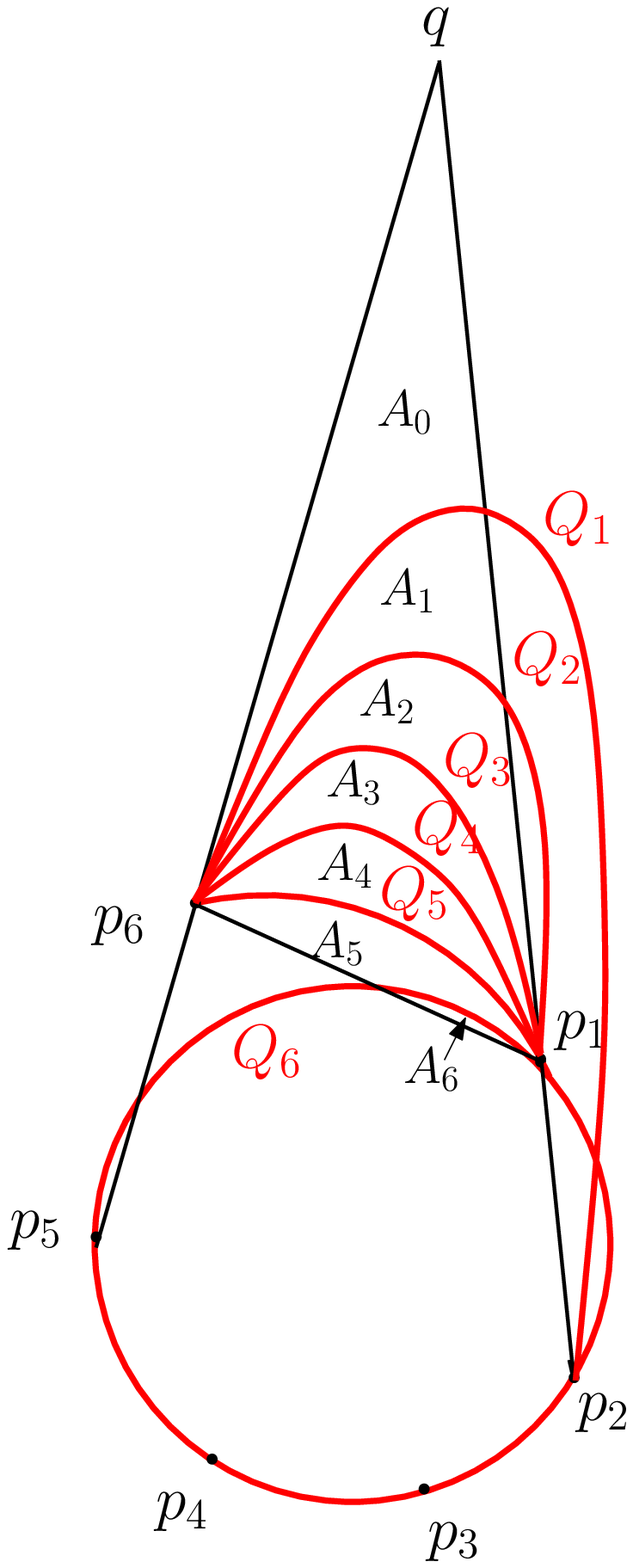}}}
   \caption{A triangle of type $A$ is divided by 6 conics $Q_{i}$ into 6 or 7 Q-regions,
$A_{i}$. For $i=1,\dots 5$, region $A_{i}$ is bounded by conics $Q_{i}$ and $Q_{i+1}$. $A_{0}$ and $A_{6}$ lie
respectively outside of $Q_{1}$ and inside $Q_{6}$.}
\label{6and7regions}
\end{figure}
The only one of these Q-regions that can be collapsed is $A_6$, so,
the monodromy group $\Aut_Q(\P_0)\cong\DD_3$ acts
on the Q-regions of types $A_i$, $0\le i <6$ and clearly, forms
6 orbits, denoted respectively by $[A_i]$.
Our choice of $p_0$ with $d(p_0)=6$ means that this point lies inside region $A_0$.
Transitivity of $\DD_3$-action on the Q-regions of type $A_0$
and impossibility for such a region to be collapsed in the
process of a Q-deformation implies connectedness of $\QC{7}{(7,0,0,0)}$.

\begin{remark}
 Placing point $p_0$ in a Q-region $A_i$, $0< i\le6$, instead of $A_0$ lead to
a new canonical cyclic ordering of the points of $\P$ (different from
 $p_0,\dots,p_6$). To recover that order, it is sufficient to know
 the two points, of dominance indices
 $0$ and $6$.
 Table \ref{table} shows how this pair of points depend on the region $A_i$.
\begin{table}[h]
\caption{The indices $d(p_j)$ in case of $p_0\in A_i$}
\centering
\label{table}
\scalebox{1}{
\begin{tabular}{c|c|c|c|c|c|c|c}
\hline%
Location of $p_0$ & $A_6$ & $A_5$ & $A_4$ & $A_3$ & $A_2$ & $A_1$ & $A_0$\\
\hline
The point of $\P$ with $d=6$  & $p_6$ & $p_6$ & $p_4$ & $p_4$ & $p_2$ & $p_2$ & $p_0$\\
\hline
The point of $\P$ with $d=0$ & $p_0$ & $p_5$ & $p_5$ & $p_3$ & $p_3$ & $p_1$ & $p_1$\\
\hline
\end{tabular}}
\end{table}
\end{remark}

\section{Q-deformation classification of hexagonal $7$-configurations}\label{hex}

\subsection{General scheme of arguments: subdivision of L-polygons into Q-regions}
In all the cases we follow the same scheme of Q-deformation classification as for
heptagonal configurations in Subsection \ref{heptagonal-Qdeformations}.
Namely, we consider a typical 7-configuration with a marked point $\P=\P_0\cup\{p_0\}$, so that
 $\P_0\in\QC61$ (a cyclic 6-subconfiguration).
In this section we assume that $p_0$ lies inside hexagon $\G_{\P_0}$, which corresponds to the case
of hexagonal configuration $\P$ (see Subsection \ref{dual}).
For a given $\P$ the number of such choices of $p_0$ is equal to $f_6(\P)$
(recall that an affine chart $(\Rp2\sm\ell)\supset\P$
in which the convex hull of $\P$ is hexagonal corresponds in the dual terms to a choice of point $\hat\ell$
inside a hexagonal component of $\widehat{\Rp2}\sm\widehat{\P}$ for the dual arrangement $\widehat{\P}$).
Since in our case $f_6(\P)=1$
we conclude that a choice of marked point $p_0$ is unique.

In the next Subsection we consider $p_0$ lying outside $\G_{\P_0}$ in one of L-polygons that correspond to
pentagonal configurations (so, we exclude previously considered cases of heptagonal and hexagonal 7-configurations).

Conics $Q_i$ passing through the points of 5-subconfigurations $\P_0\sm\{p_i\}$, $i=1,\dots,6$,
can subdivide an L-polygon into several {\it Q-regions} like in Subsection \ref{heptagonal-Qdeformations},
and our aim is to analyze which of these regions cannot be contracted in a process of Q-deformation,
and how the monodromy group $\Aut_Q(\P_0)$ does act on them.

We always choose a cyclic order of points
$p_1,\dots,p_6\in\P_0$ so that $p_1$ is dominant (then $p_3$ and $p_5$ are dominant too, whereas $p_2$, $p_4$, $p_6$ are subdominant).
Then
conics $Q_2$, $Q_4$, and $Q_6$ contain hexagon $\G_{\P_0}$ inside,
whereas $Q_1$, $Q_3$, and $Q_5$ intersect the internal L-polygons of $\P_0$, see
Figure~\ref{qregions}.
\begin{figure}[h!]
\centering
{\scalebox{0.4}{\includegraphics{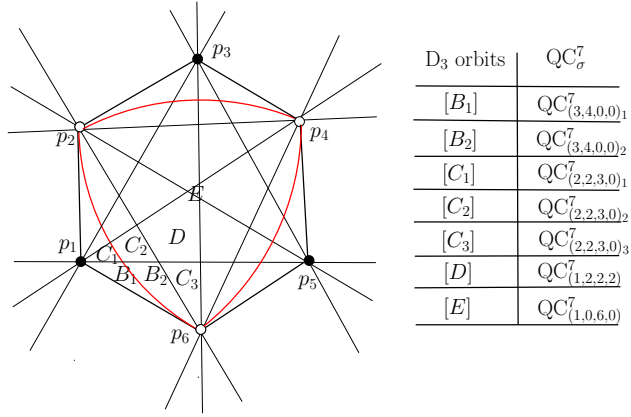}}}
\caption{The seven $\DD_3$-orbits on internal $Q$-regions associated to a hexagonal configuration in $QC^{7}$}
\label{qregions}
\end{figure}
\begin{remark}
On Figure \ref{qregions} conics $Q_1$, $Q_3$, $Q_5$ do not intersect the sides of the hexagon.
If they do intersect, then the shape of Q-regions of types $B_1$ and $B_2$ may change,
which changes Figure \ref{qregions} a bit, but not essentially for our arguments.
\end{remark}

\subsection{$\DD_3$-orbits}
The internal L-polygons of types $D$ and $E$ are obviously contained inside these conics and only
L-polygons of types $B$ and $C$ are actually subdivided into Q-regions.
Namely, the latter L-polygons are subdivided into Q-regions $B_1$, $B_2$ and respectively
$C_1$, $C_2$, $C_3$ as it is shown.
Next, we can easily see that
monodromy group $\Aut_Q(\P_0)\cong\DD_3$ acts transitively on the Q-regions of each type, which
gives seven $\DD_3$-orbits: $B_1$, $B_2$, $C_1$, $C_2$, $C_3$, $D$ and $E$.

\subsection{Deformation classification in the cases of non-collapsible Q-regions}\label{non-collabsible}
Note that triangle $E$ is the only internal Q-region of ${\P_0}$ that can be collapsed by a Q-deformation.
Thus, any pair of hexagonal $7$-configurations $\P'$ and $\P''$ whose marked points, $p_0'$ and $p_0''$ belong
to the same $\DD_3$-orbit of the internal Q-regions different from $E$ can be connected by
a Q-deformation. Namely, we start with a Q-deformation between $\P'_0=\P'\sm\{p_0'\}$ and $\P''_0=\P''\sm\{p_0''\}$
that transforms the Q-region containing $p_0'$ to the one containing $p_0''$
and extend this deformation to the seventh points using non-contractibility of the given type of Q-regions.
This yields Q-deformation classes $\QC7{(3,4,0,0)_1}$, $\QC7{(3,4,0,0)_2}$, $\QC7{(2,2,3,0)_1}$, $\QC7{(2,2,3,0)_2}$,
$\QC7{(2,2,3,0)_3}$, and $\QC7{(1,2,2,2)}$ that correspond respectively to the $\DD_3$-orbits of types
$B_1$, $B_2$, $C_1$, $C_2$, $C_3$, and $D$, see Figure \ref{qregions}.

\subsection{The case of Q-region $E$}\label{collapsibleE}
Consider a subset $\widetilde{\LC61}\subset\LC61$ formed by typical hexagonal 6-configurations $\P=\{p_1,\dots,p_6\}$
whose {\it principal diagonals} $p_1p_4$, $p_2p_5$, and $p_3p_6$ are not concurrent, or in the other words, whose
L-polygon $E$ is not collapsed.
Connectedness of space $\QC7{(1,0,6,0)}$ formed by 7-configurations $\P=\P_0\cup\{p_0\}$ with $p_0$
placed in the Q-region $E$ would follow from connectedness of $\widetilde{\LC61}$.

\begin{proposition}\label{specialqdeformation}
Space $\widetilde{\LC61}$ is connected.
\end{proposition}

\begin{proof}
Given a pair of configurations, $\P^{i}\in \widetilde{\QC6{1}}$, $i=0,1$,
we need to connect them by some deformation $\P^{t}\in \widetilde{\QC6{1}}$, $t\in[0,1]$.
Let us choose a cyclic numeration of points, $p^{i}_{1},\dots,p^{i}_{6}\in\P^{i}$, $i=0,1$, so that
$p^{i}_1$ are dominant.
 At the first step, we can achieve that the triangular Q-regions ``E''
 of the both configurations coincide, so that
 the dominant and subdominant points in $\P^i$ go in the same order, as it is shown on Figure \ref{centraltrianglesame}(a),
that is, points $p_1^i\in\P^i$, $i=0,1$, lie on the
ray that is the extention of side $XY$ (then the other rays extending the sides if triangle $XYZ$
are also of the same color).

\begin{figure}[h!]
\centering
\subfigure[(a)\qquad\qquad\ ]{\scalebox{0.35}{\includegraphics{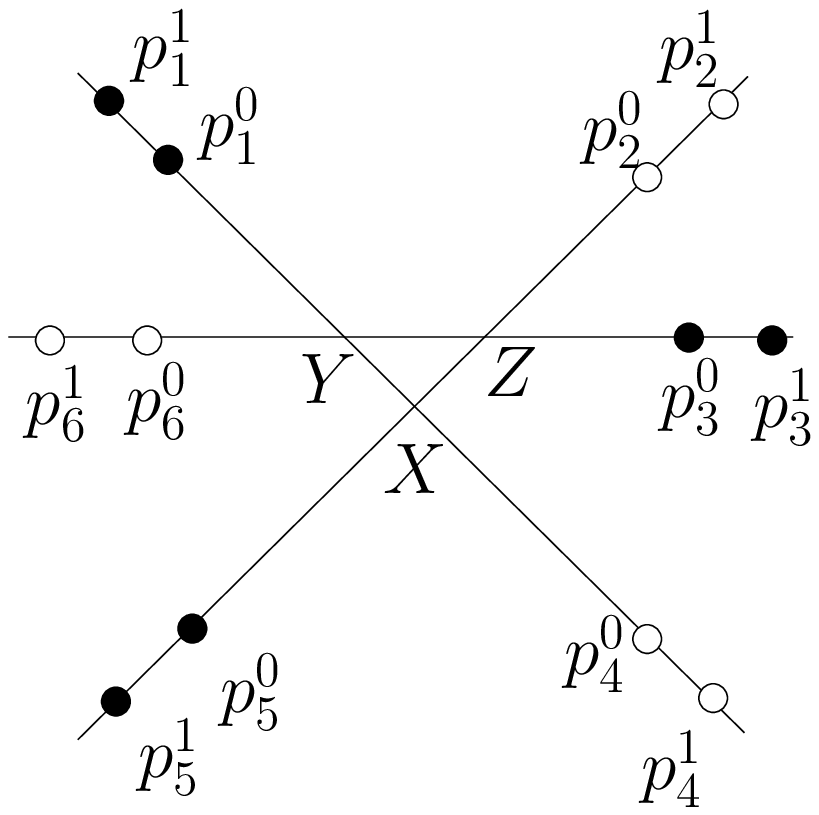}}\qquad\qquad}
\subfigure[\qquad\ \ (b)\qquad]{\scalebox{0.4}{\includegraphics{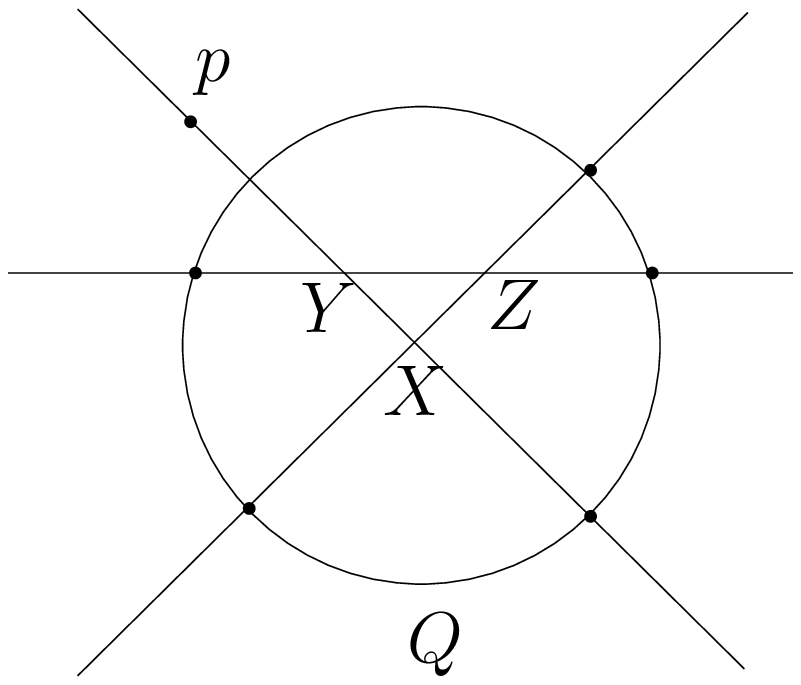}}}
\caption{(a) The affine hexagons  $\P^0$ and $\P^1$ have common principal diagonals and
the dominant point $p_1^i$, of $\P^i$, $i=0,1$, lie on the same continuation of side $XY$.
Then the other dominant (subdominant) points $p_k^i\in\P^i$ (being numerated in the same direction)
lie also in an alternating way
on the corresponding extensions of the sides of $XYZ$
(here the mutual position of each pair $p_k^0$ and $p_k^1$, $k=1,\dots,6$, on the corresponding ray is not essential).
\newline(b) Conic $Q$ containing triangle $\XYZ$ inside and a point $p$ outside.}
\label{centraltrianglesame}
\end{figure}

This can be done by a projective transformation sending the diagonals $p_1^0p_4^0$, $p_2^0p_5^0$, $p_3^0p_6^0$,
and the {\it infinity line}, $L_\infty$, (that pass in the complement of $\G_{P^0}$) to the corresponding diagonals and the ``infinity line''
for configuration $\P^1$
(existence of a deformation is due to connectedness of $PGL(3,\R)$).
If the mutual positions of the dominant and subdominant points on the lines in $\P^0$ and $\P^1$ will differ,
then it can be made like on Figure \ref{centraltrianglesame}(a) by a projective transformation that permutes
the three diagonals while preserving $L_\infty$.
(One can also use flexibility of the initial numeration of vertices in $\P^1$).

Fixing a triangle $\XYZ\subset \R^{2}=\Rp2\sm L_{\infty}$,
let us denote by $\QC{6}{E,\XYZ}$ the subspace of $\QC{6}{1}$ consisting of hexagonal $6$-configurations
whose dominant points lie on the affine rays that are continuations of sides $XY$, $YZ$, and $ZX$, and subdominant points
lie on the continuations of $YX$, $ZY$, and $XZ$, as it is shown on Figure \ref{centraltrianglesame}(a).

The final step of the proof is connectedness of $\QC{6}{E,\XYZ}$.

\begin{lemma} \label{connect}
For a fixed triangle
$\XYZ\subset\R^2=\Rp2\sm L_{\infty}$, the configuration space $\QC{6}{E,\XYZ}$ is connected.
\end{lemma}
\begin{proof}
Consider $\P\in\QC{6}{E,\XYZ}$, $\P=\{p_1,\dots,p_6\}$, where
point $p_1$ is dominant one lying on $XY$.
Consider conic $Q$ passing trough $p_2,\dots,p_6$.
Triangle $\XYZ$ lies inside $Q$ and point $p_1$ lies outside.
This gives a one-to-one correspondence between $\QC{6}{E,\XYZ}$ and the space of pairs $(Q,p)$, where
$Q$ is an ellipse containing $XYZ$ inside and $p$ is a point on the continuation of $XY$ lying outside $Q$
(see Figure \ref{centraltrianglesame}(b)).
 The space of such ellipses is connected (and in fact, contractible), and the projection of $\QC{6}{E,\XYZ}$
to this space is fibration with a contractible fiber. Thus, $\QC{6}{E,\XYZ}$ is connected (and in fact, is contractible).
\end{proof}
\vskip-3mm
\end{proof}
\subsection{Decoration of the adjacency graphs for hexagonal typical 7-configurations}\label{decoadj}
Figure \ref{QHEXED7} shows the adjacency graphs of typical hexagonal configurations $\P$ endowed
\begin{figure}[h!]
  \centering
  \subfigure[$QC^{7}_{(3,4,0,0)_1}$]{\scalebox{0.33}{\includegraphics{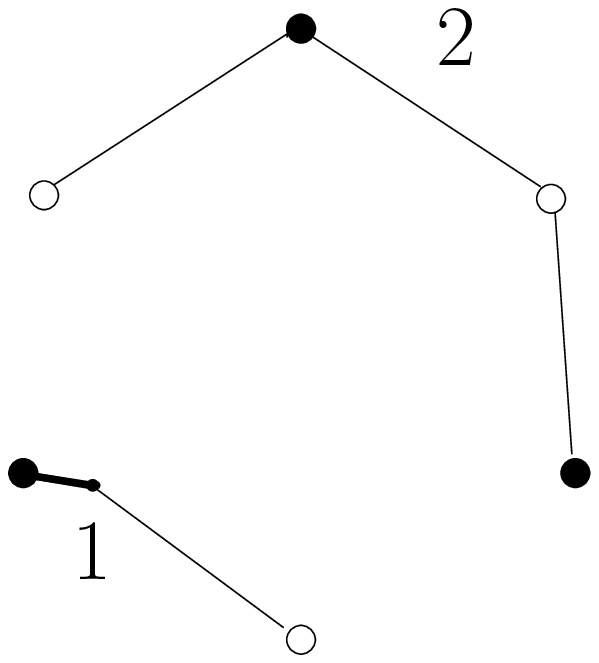}}}\hspace{1cm}
	\subfigure[$QC^{7}_{(3,4,0,0)_2}$]{\scalebox{0.33}{\includegraphics{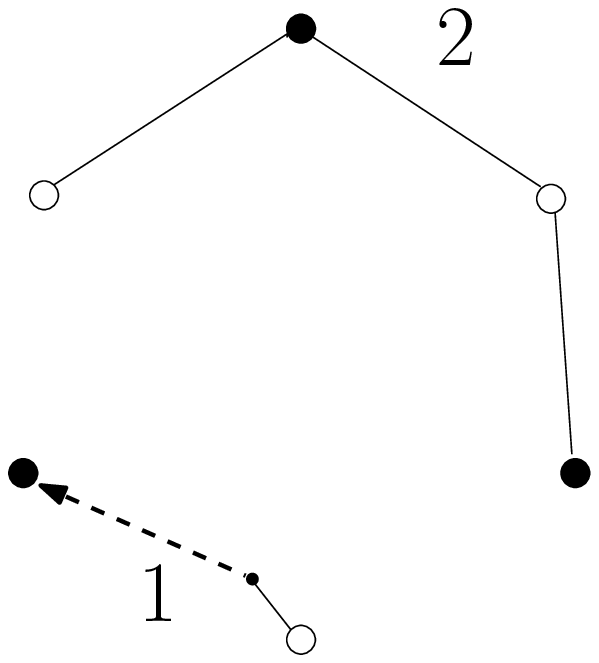}}}\\
  \subfigure[$QC^{7}_{(2,2,3,0)_1}$]{\scalebox{0.33}{\includegraphics{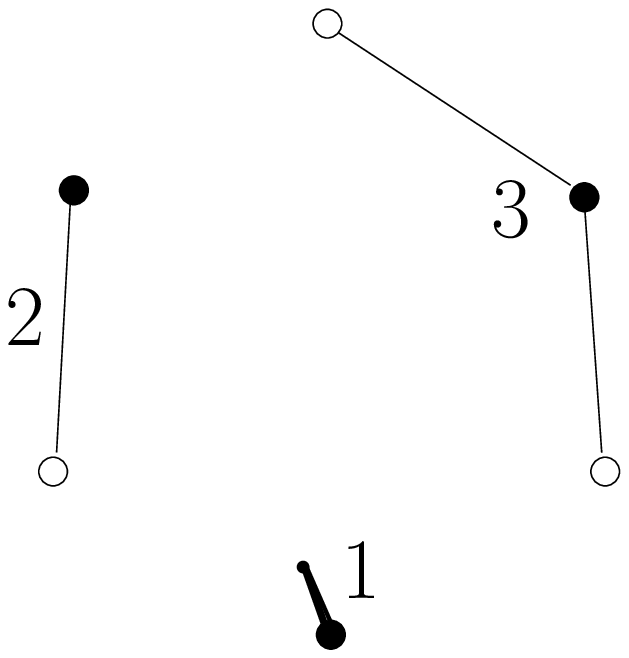}}}\hspace{1cm}
	\subfigure[$QC^{7}_{(2,2,3,0)_2}$]{\scalebox{0.33}{\includegraphics{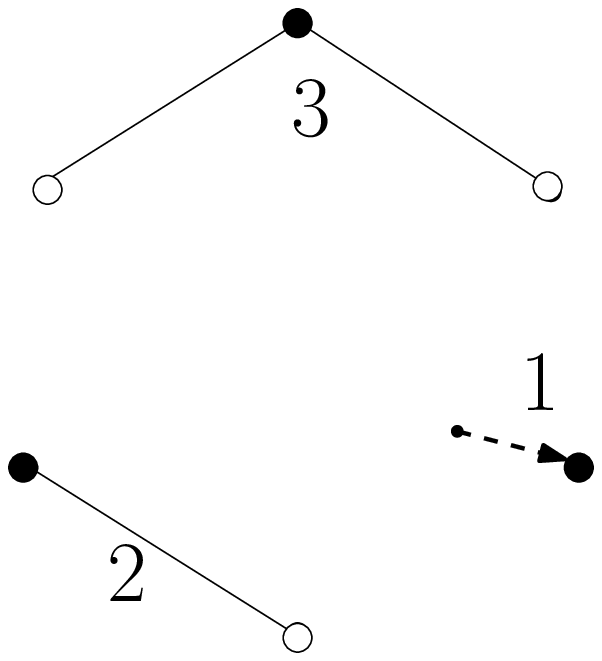}}}\hspace{1cm}
	\subfigure[$QC^{7}_{(2,2,3,0)_3}$]{\scalebox{0.33}{\includegraphics{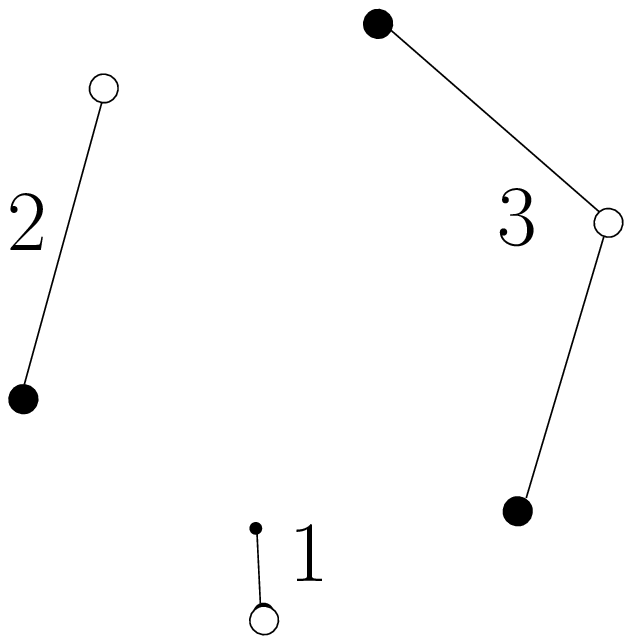}}}\\
  \subfigure[$QC^{7}_{(1,2,2,2)}$]{\scalebox{0.33}{\includegraphics{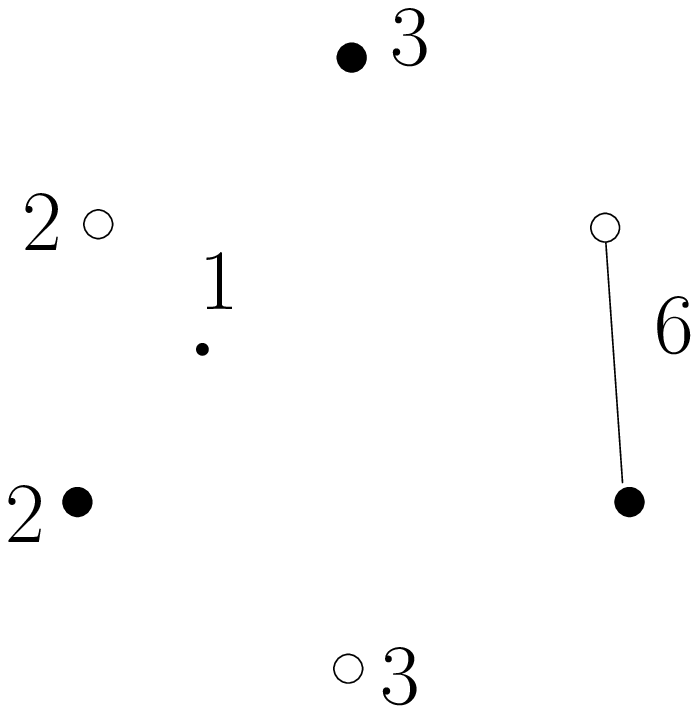}}}\hspace{1cm}
	\subfigure[$QC^{7}_{(1,0,6,0)}$]{\scalebox{0.33}{\includegraphics{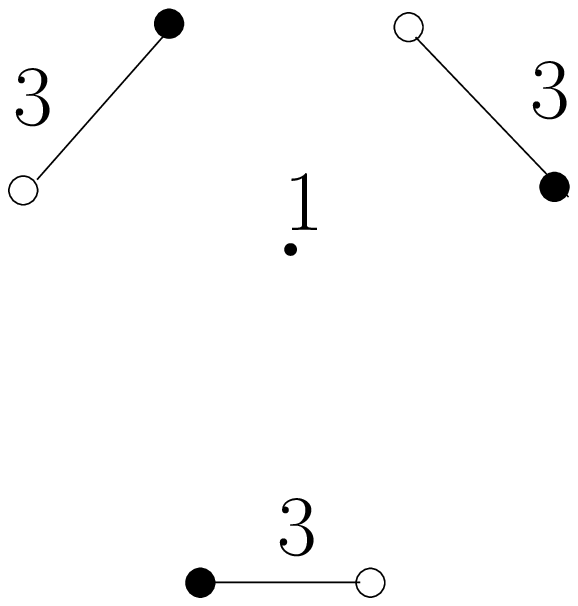}}}
	\caption[$Q$-deformation classes of hexagonal $7$-configurations]{$Q$-deformation classes of hexagonal $7$-configurations}
	\label{QHEXED7}
\end{figure}
additionally with the {\it vertex coloring} for $\P_0=\P\sm\{p_0\}$
as in Subsection \ref{QGC} (black for dominant and white for subdominant points) and with
the {\it edge decoration} for a connected component of $\G_{\P}$ labeled by $1$.
Namely, such an edge $[p_ip_j]$ is {\it thin} if $p_i$ and $p_j$ lie inside conic $Q_{ij}$,
{\it thick} if they lie outside, or {\it dotted and directed}
from $p_i$ to $p_j$ if $p_i$ lies inside and $p_j$ lies outside
(like is shown on Figure \ref{edecoration}).
Such decoration lets us distinguish Q-deformation types of hexagonal configurations.

\subsection{Coloring of vertices in the case of pentagonal 7-configurations with $\s_1>0$}
Recall that a pentagonal typical 7-configuration $\P$ has precisely $\s_1$
vertices $p\in\P$ such that $\P_0=\P\sm{p}\in\QC61$. Moreover, for pentagonal configurations $\s_1$ is either $0$ or $1$
(see Table 1). So, in the case $\s_1>0$ (that is $\s_1=1$) 
considered in the next section such a vertex $p$ is unique, and
we can (and will) color the six vertices of $\P_0$ according to their dominancy as before.

\section{Pentagonal $7$-configurations, the case of $\sigma_{1}>0$}\label{pen}

\subsection{$\DD_3$-orbits of types $G$ and $H$}\label{G-H}

A pentagonal typical $7$-configuration with $\sigma_{1}>0$ (and thus, $\sigma_1=1$)
can be presented like in the previous section
as $\P=\P_0\cup\{p_0\}$, where $P_0\in\QC61$. The difference is that now point $p_0$ lies outside hexagon $\G_{\P_0}$.
More precisely, $p_0$ should lie in an L-polygon of type $F$, or $G$, or $H$, since the other types of L-polygons
correspond either to heptagonal (the case of L-polygons of type $A$) or hexagonal (the case of types $I$ and $J$) configurations
$\P$ that were analyzed before.

The first crucial observation is that none of the conics $Q_i$, $i=1,\dots,6$ can intersect these three types of external polygons, and therefore, such L-polygons are not subdivided into Q-regions like in the case of types $A$, $B$ and $C$.
It is clear from Figure \ref{pentagonregions}: conics $Q_i$ should lie in the shaded part that is formed by
L-polygons of types $A$, $I$ and $J$.
\begin{figure}[h!]
\centering
\subfigure[]{\scalebox{0.24}{\includegraphics{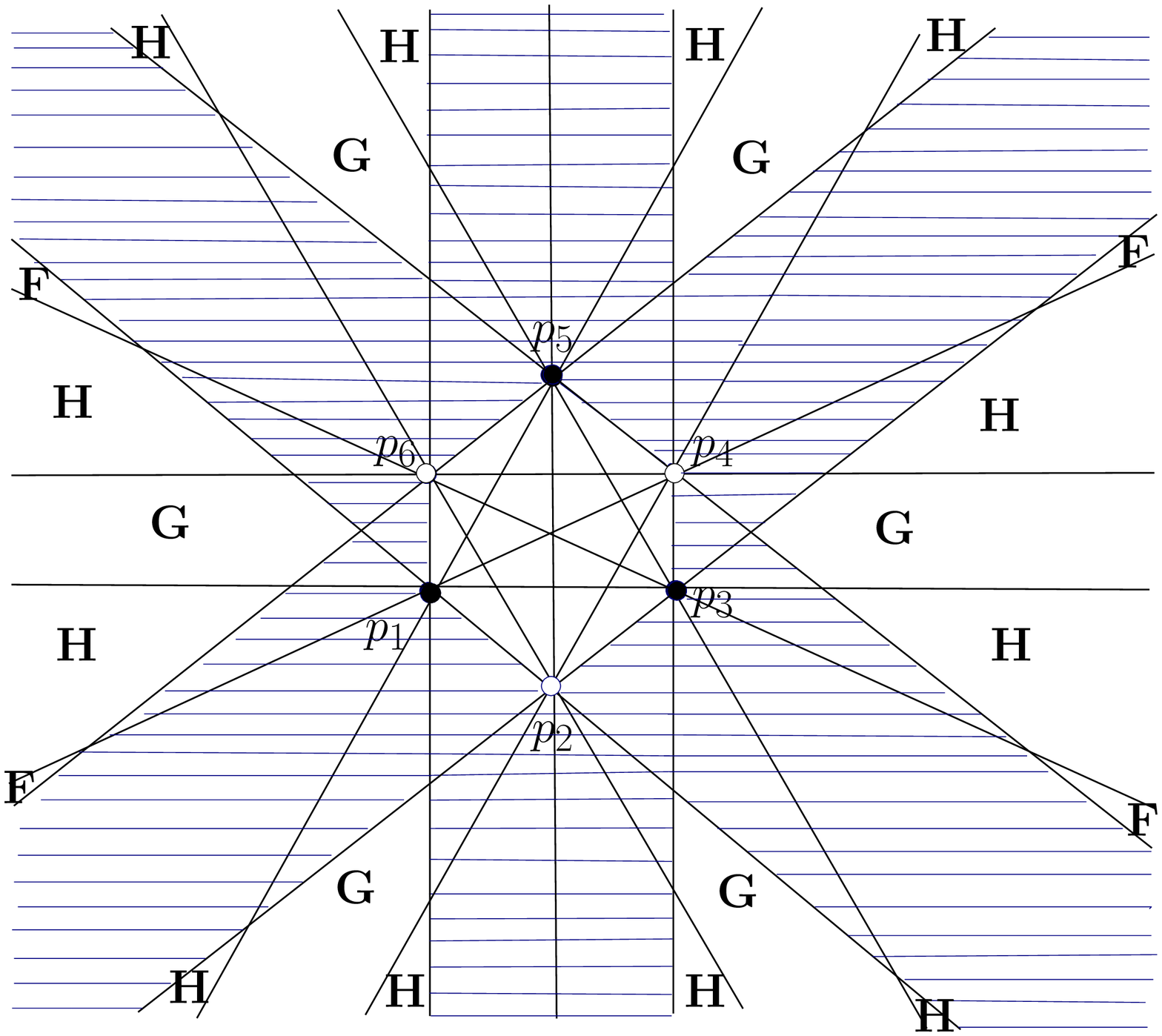}}}\;\;
\subfigure[]{\scalebox{0.3}{\includegraphics{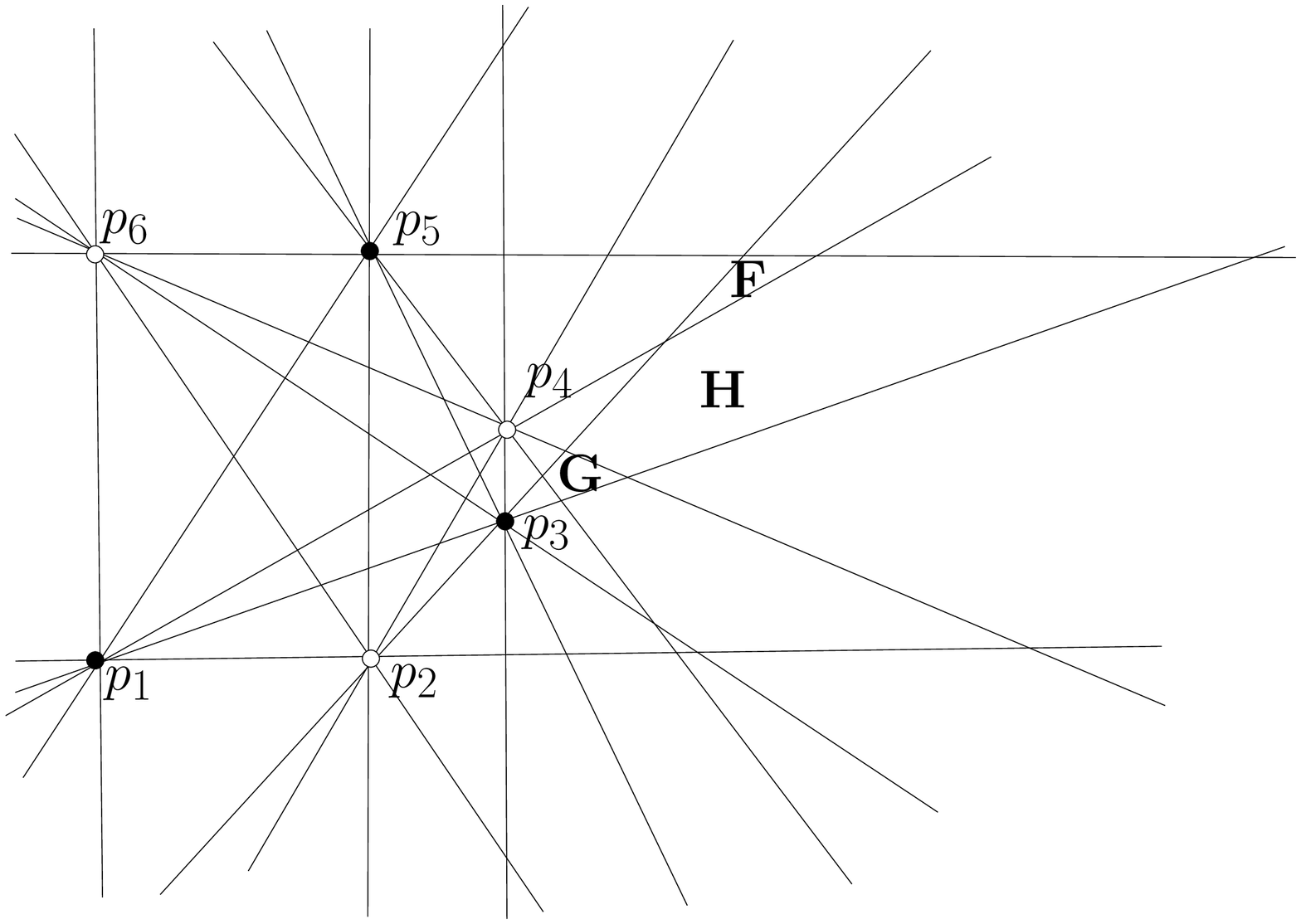}}}
\caption
{External $Q$-regions representing pentagonal $7$-configurations with $\sigma_{1}=1$.
$\DD_3$-action is clear on the left, and the shape of L-polygons $G$, $H$, $F$ on the right.}
\label{pentagonregions}
\end{figure}

The second observation is that L-polygons of each type, $F$, $G$, or $H$, form a single orbit with respect to
the action of monodromy group $\Aut_Q(\P_0)=\DD_3$.

The third evident observation is that L-polygons of types $G$ and $H$ cannot be contracted by a Q-deformation
(as they cannot be contracted even by an L-deformation).
 Together these observations imply
that the corresponding to L-polygon types $G$ and $H$
(see Table~\ref{spectrandcodes}) configuration spaces $\QC7{(1,4,2,0)}$ and $\QC7{(1,2,4,0)}$
are connected, and thus, are Q-deformation components.

\subsection{The case of L-polygons of type $F$}\label{F-polygons}
\begin{proposition}\label{F-connectedness}
The configuration space $\QC{7}{(1,6,0,0)}$ is connected, or equivalently, L-deformation component $\LC7{(1,6,0,0)}$
that correspond to L-polygon of type $F$ contains a unique Q-deformation component.
\end{proposition}

\begin{proof}
A configuration $\P\in\QC{7}{(1,6,0,0)}$ has a unique distinguished point $p_0$, such that
$\P=\P_0\cup\{p_0\}$, where $\P_0\in\QC61$ and $p_0$ lies in the L-polygon of type $F$.
Such polygon is a triangle whose vertices we denote by $X$, $Y$, and $Z$ using the following rule.
By definition of $F$-type polygon,
one of its supporting lines should be a principal diagonal passing through two opposite vertices
of hexagon $\G_{\P_0}$. We can choose a cyclic numeration of points
$p_1,\dots, p_6\in\P$ so that these opposite vertices are $p_1$ and $p_4$, and $p_1$ is a dominant point (then $p_3$, $p_5$
are also dominant, and $p_2$, $p_4$, $p_6$ are subdominant).
Two vertices of the triangle on the line $p_1p_4$ are denoted by $X$ and $Y$ in such an order that
$X$, $Y$, $p_1$, $p_4$ go consecutively on this line, like it is shown on Figure \ref{xyz}(a), and the third point
of the triangle is denoted by $Z$.
 The direction of cyclic numeration of points $p_i$ can be also chosen so that points
$p_2$, $p_3$ lie on the line $XZ$ and $p_5$, $p_6$ on $YZ$ (see Figure \ref{xyz}(a)).
\begin{figure}[h]
\centering
\subfigure[(a)\ \ \qquad\qquad\qquad]{\scalebox{0.3}{\includegraphics{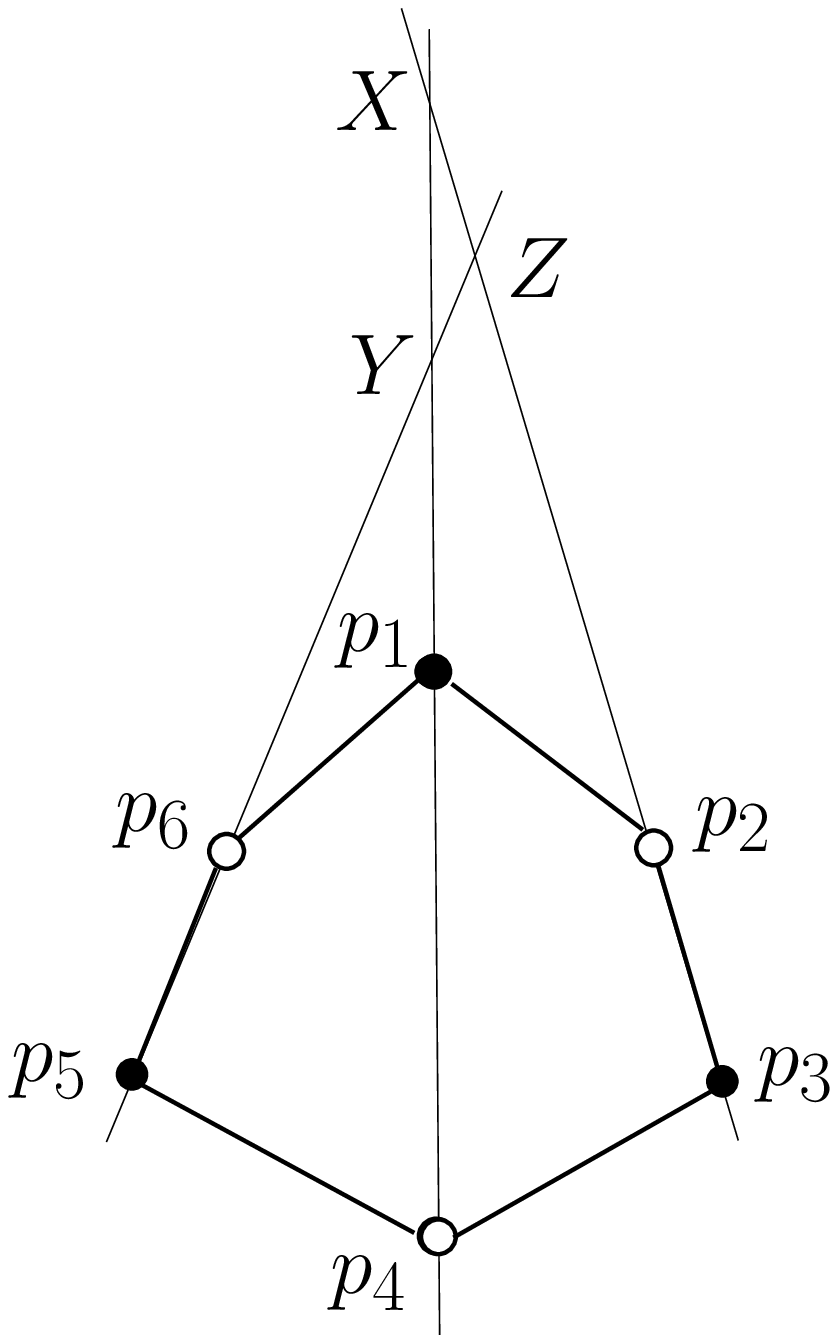}}\qquad\qquad\qquad}
\subfigure[(b)\ ]{\scalebox{0.3}{\includegraphics{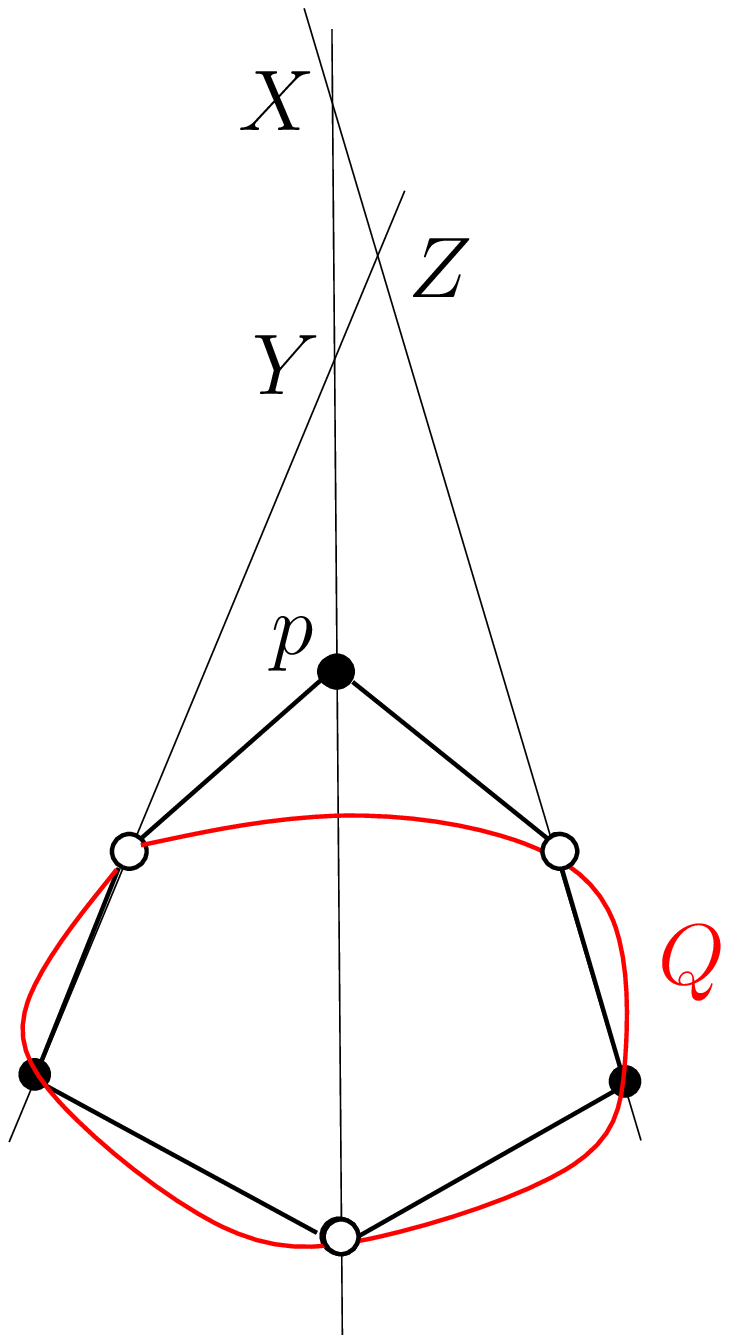}}}
\caption{}
\label{xyz}
\end{figure}

By a projective transformation we can map a triangle $\XYZ$ to any other triangle on $\Rp2$,
so, in what follows we suppose that triangle $\XYZ$ is fixed and denote by $\QC{6}{F,\XYZ}\subset\QC61$
the subspace formed by typical cyclic configurations having $\XYZ$ as its L-polygon of type $F$
and having a cyclic numeration of points $p_1,\dots,p_6\in\P$
satisfying the above convention.

Then, Proposition \ref{F-connectedness} follows from connectedness of $\QC{6}{F,\XYZ}$.

\begin{lemma}\label{F-connect}
For a fixed triangle $\XYZ$, the configuration space  $\QC{6}{F,\XYZ}$ is connected.
\end{lemma}

\begin{proof}
Using the same idea as in Lemma \ref{connect}, we associate with a configuration $\P\in\QC{6}{F,\XYZ}$ a pair
$(Q,p)$, where $p$ is the dominant point of $\P$ on the line $XY$ (that is $p_1$ in the notation used above)
and $Q$ is the conic passing through the other points of $\P$.
 Note that $\P$ can be recovered from pair $(Q,p)$ associated to it in a unique way.
Position of $Q$ can be characterized by the conditions that triangle
$\XYZ$ lie outside $Q$ and the
lines $XY$, $XZ$, $YZ$ intersect conic $Q$ at two points, so that
the chord of conic $Q$ that is cut by line
$XY$ lies between the two other chords that are cut by $XZ$ and $YZ$ (see Figure \ref{xyz}(b)).

The set of conics satisfying these requirements is obviously connected (and in fact, is contractible).
 For each conic $Q$ like this, there is some interval on the line $XY$ (see Figure \ref{xyz})
 formed by points $p$ such that $(Q,p)$ is associated to some $\P\in\QC{6}{F,\XYZ}$.
Thus, the set of such pairs $(Q,p)$, or equivalently $\QC{6}{F,\XYZ}$, is also connected.
\end{proof}
\vskip-3mm
\end{proof}

\subsection{Proof of Theorem ~\ref{14classes}}\label{proof}
We have shown in Subsection \ref{3trivial} connectedness of three components
$\QC7{(0,4,3,0)}$, $\QC7{(0,6,1,0)}$, and $\QC7{(0,3,3,1)}$ of pentagonal 7-configurations with $\s_1=0$.
In Subsection \ref{heptagonal-Qdeformations} we have shown connectedness of the component $\QC7{(7,0,0,0)}$ formed by heptagonal typical 7-configurations, and in Section \ref{hex} found seven connected components formed by hexagonal 7-configurations.
The remaining 3 cases of pentagonal configurations with $\s_1=1$ were analyzed in Subsections \ref{G-H} and \ref{F-polygons}.
\qed

\section{Concluding Remarks}\label{concluding}

\subsection{Real Schl\"afli double sixes of lines}
By blowing up $\mathbb{P}^2$ at the points of a typical 6-configuration $\P\subset \mathbb{P}^2$ we obtain a del Pezzo surface $X_\P$
of degree 3 that can be realized by anti-canonical embedding as a cubic surface in $\mathbb{P}^3$.
The exceptional curves of blowing up
form a configuration of six skew lines $\mathcal L_{\P}\subset X_\P\subset \mathbb{P}^3$ that
is nothing but a half of Schl\"afli's double six of lines, and we call below such $\mathcal L_{\P}$
the {\it skew six of lines} represented by $\P$.
 In the real setting, for $\P\subset\Rp2$, cubic surface $X_\P$ is real and {\it maximal}, where the latter means by definition that the real locus $\R X_\P\subset X_\P$ is homeomorphic to $\Rp2\#6\Rp2$.
 The four deformation classes of typical 6-configurations give
four types of real skew sixes of lines: {\it cyclic, bicomponent, tricomponent and icosahedral}.
It was observed in \cite{Z} that the complementary real skew six of lines
(that forms together with $\mathcal L_{\P}$ a real double six on $X_\P$) has the same type as a given one,
and so, we can speak of the {\it four types of real double sixes of lines}.

It was shown by V.\,Mazurovski (see \cite{DV})
that there exist 11 {\it coarse deformation classes} of
six skew line configurations in $\Rp3$: here {\it coarse} means that deformation equivalence is combined
with projective (possibly orientation-reversing) equivalence, for details see \cite{DV}. Among these 11 classes, 9 can be realized
by so called {\it join configurations}, $J_{\tau}$, that can be presented by permutations
$\tau\in S_6$ as follows. Fixing consecutive points $p_1,\dots, p_6$ and $q_1,\dots,q_6$ on a pair of auxiliary skew lines, $L^p$ and $L^q$ respectively,
we let $J_{\tau}=\{L_1,\dots,L_6\}$, where line $L_i$ joins $p_i$ with
$q_{\tau(i)}$, $i=1,\dots,6$. We denote such a configuration (and sometimes its coarse deformation class) by $J_\tau$.
 The remaining two coarse deformation classes among 11
 cannot be represented by join configurations $J_\tau$; these two classes are denoted in \cite{DV}
 by $L$ and $M$.
As it is shown in \cite{Z},
the cyclic, bicomponent, and tricomponent coarse deformation classes of real skew sixes $\Cal L_{\P}$ are
realized as $J_\tau$, where $\tau$ is respectively $(12\dots6)$, $(123654)$, and $(214365)$,
where $\tau$ is recorded as $(\tau(1)\dots\tau(6))$ (see Figure ~\ref{CAC6graph}).
The icosahedral coarse deformation class corresponds to the class $M$ from \cite{DV}.
\begin{figure}[h!]
\centering
{\scalebox{0.35}{\includegraphics{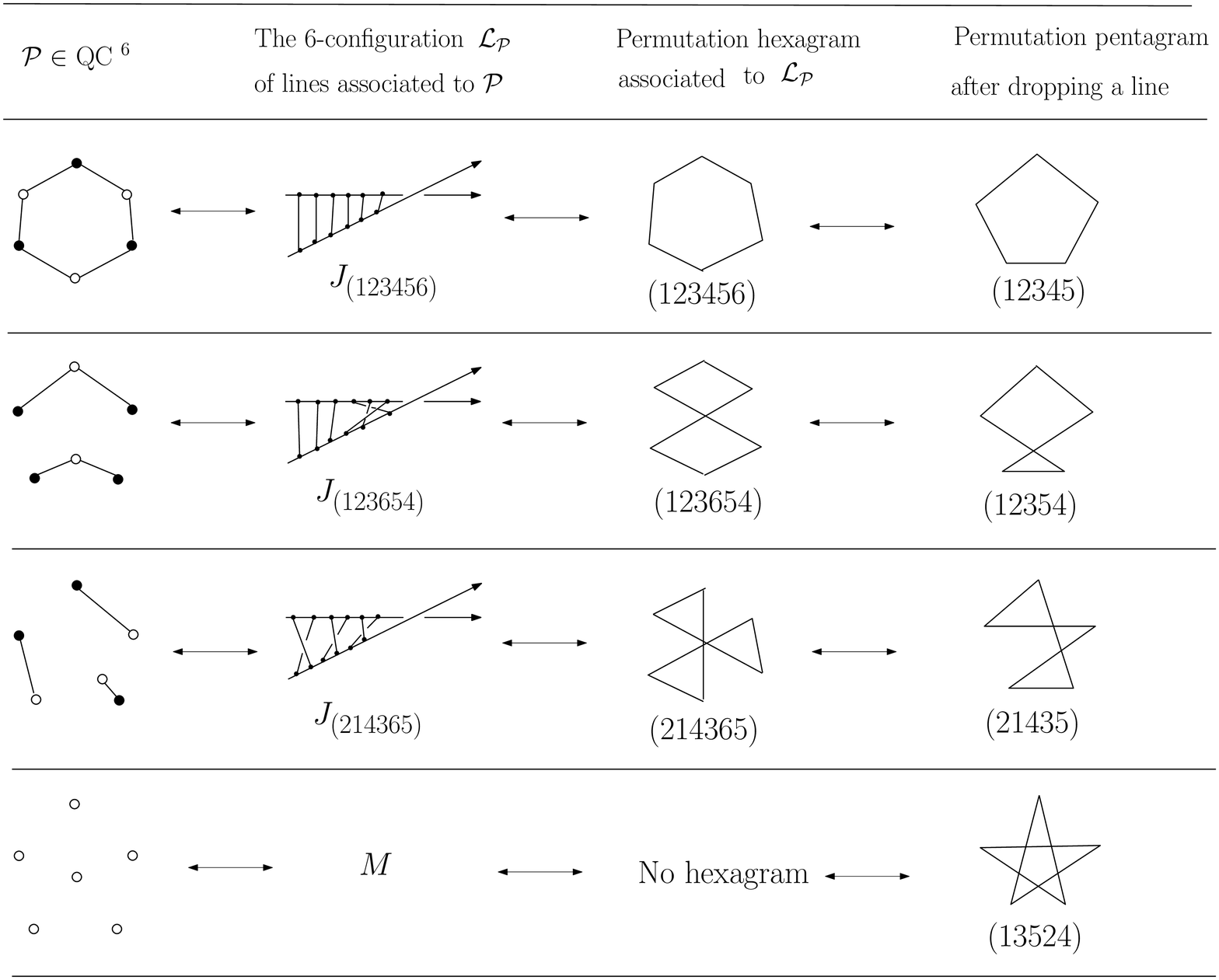}}}
\caption{Four classes of simple $6$-configurations, the corresponding real skew sixes of lines in $\Rp3$, with their permutation hexagrams and pentagrams}
\label{CAC6graph}
\end{figure}
\subsection{Permutation Hexagrams and Pentagrams}
A change of cyclic orderings of points $p_i$, $q_i$ on lines $L^p$ and $L^q$ clearly
does not change  the coarse deformation class of $J_\tau$.
 In the other words, the coarse deformation class of $J_\tau$ is an invariant of the orbit
$[\tau]\in S_6/(\DD_6\times\DD_6)$ of $\tau\in S_6$ with respect to the left-and-right multiplication action of $\DD_6\times\DD_6$ in $S_6$
for the dihedral subgroup $\DD_6\subset S_6$.

With a permutation  $\tau\in S_n$ we associate a diagram $D_\tau$ obtained by connecting
cyclically ordered vertices $v_1\dots v_n$ of a regular $n$-gon by diagonals $v_{\tau(i)}v_{\tau(i+1)}$, $i=1,\dots,n$,
(here, $\tau(n+1)=\tau(1)$).
 Then ``the shape of $D_\tau$'' characterizes class $[\tau]\in S_n/(\DD_n\times\DD_n)$,
 see Figure \ref{CAC6graph} for the
{\it hexagrams} representing the cyclic, bicomponent, and tricomponent permutation orbits $[\tau]$, namely, $[123456]$, $[123654]$ and $[214365]$.

By dropping a line from a real skew six $\Cal L$ we obtain {\it a real skew five}, $\Cal L'$, that can be realized
similarly, as a join configuration $J_\tau$ for $\tau\in S_5$. It was shown in \cite{Z} that the class $[\tau]\in S_5/(\DD_5\times\DD_5)$
does not depend on the line in $\Cal L$ that we dropped, including the case of icosahedral real double sixes,
see the corresponding  {\it pentagrams} $D_{[\tau]}$ on Figure \ref{CAC6graph}.

\subsection{Real Aronhold sets}
By blowing up the points of a typical 7-configuration, $\P\subset\Rp2$,
we obtain a non-singular real del Pezzo surface $X_\P$ of degree $2$
with a configuration $\Cal L_\P$ of $7$ disjoint real lines (the exceptional curves of blowing up).
The anti-canonical linear system maps $X_\P$ to a projective plane as a double covering branched along
a non-singular real quartic,
whose real locus has 4 connected components.
Each of the 7 lines of $\Cal L_{\P}$ is projected to a real bitangent to this quartic,
and the corresponding arrangement of 7 bitangents is called an {\it Aronhold set}.

The 14 Q-deformation classes of typical 7-configurations yield 14 types of real Aronhold sets,
which were described in \cite{Z}, see Appendix.

Among various known criteria to recognize that real bitangents $L_i$, $i=1,\dots,7$,
to a real quartic form an Aronhold set, topologically the most practical one is perhaps
possibility to color the two line segments between the tangency points on
each $L_i$ in two colors, so that at the intersection points $L_i\cap L_j$,
the corresponding line segments of $L_i$ and $L_j$ are colored differently.
Such colorings are indicated on the Figures in the Appendix.

\subsection{Real nodal cubics}
In \cite{S1}, Fiedler-Le-Touz\'e
analyzed real nodal cubics, $C_i$, passing through the points $p_0,\dots,p_6\in\P$
of a heptagonal configuration,
$\P\in\QC7{7,0,0,0}$, and having a node at one of the points $p_i\in\P$, and described in which order the points of $\P$
may follow on the real locus of $C_i$ (see Figure \ref{nodalcubic}).

\begin{figure}[h!]
 \centering
    \subfigure[$C_0$]{\scalebox{0.25}{\includegraphics{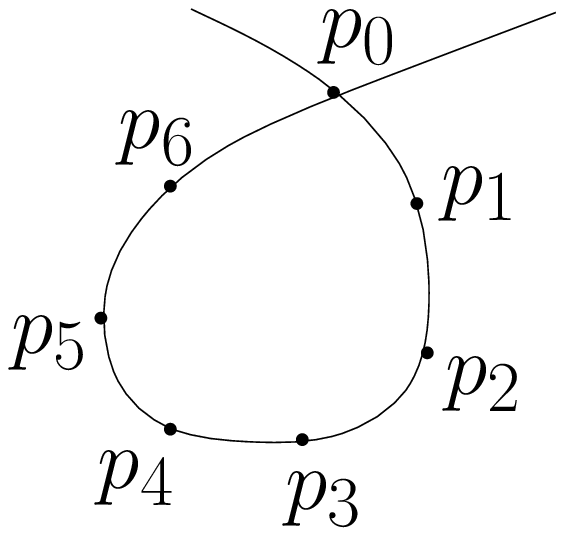}}}\hspace{0.2cm}
    \subfigure[$C_1$]{\scalebox{0.25}{\includegraphics{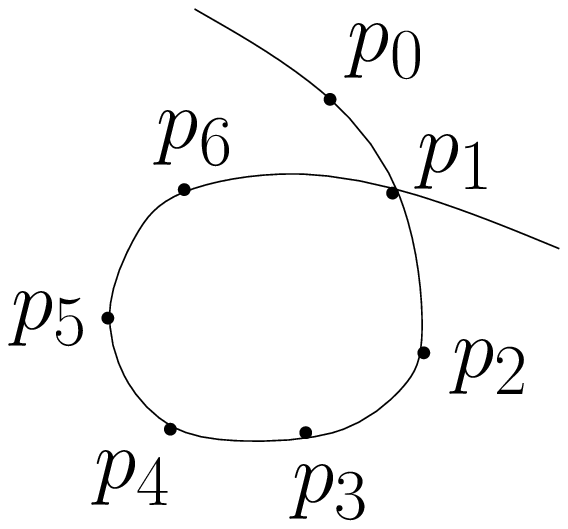}}}\hspace{0.2cm}
    \subfigure[$C_2$]{\scalebox{0.25}{\includegraphics{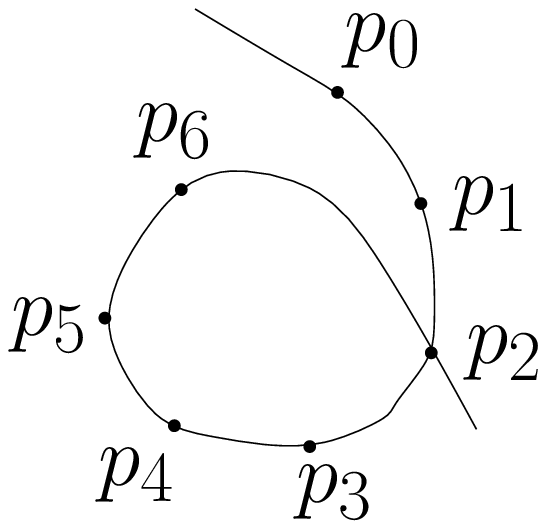}}}\hspace{0.2cm}
    \subfigure[$C_3$]{\scalebox{0.25}{\includegraphics{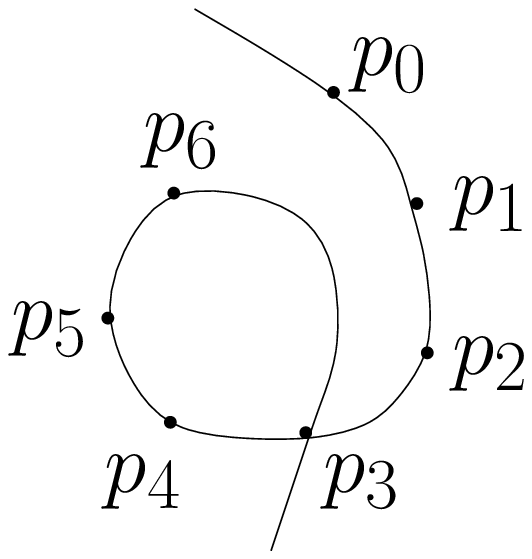}}}\hspace{0.2cm}
    \subfigure[$C_4$]{\scalebox{0.25}{\includegraphics{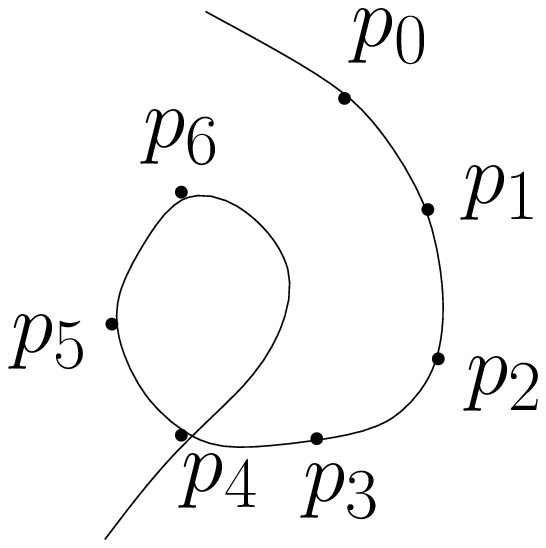}}}\hspace{0.2cm}
    \subfigure[$C_5$]{\scalebox{0.25}{\includegraphics{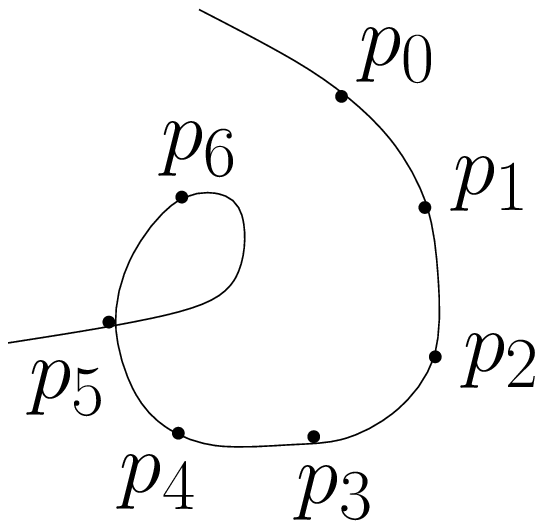}}}\hspace{0.2cm}
    \subfigure[$C_6$]{\scalebox{0.3}{\includegraphics{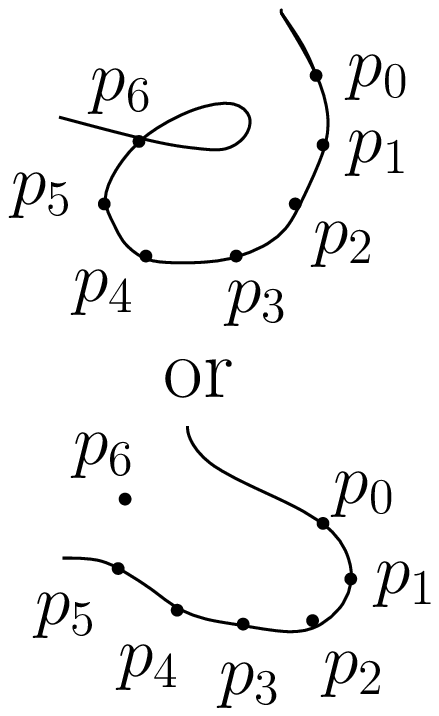}}}
  \caption{Cubics $C_i$, $i=0,\dots,6$, passing through canonically ordered points
  $p_0,\dots,p_6$ of $\P\in\QC7{(7,0,0,0)}$ and having a node at $p_i\in\P$
 }\label{nodalcubic}
\end{figure}
Recently, a similar analysis was done for the other types of 7-configurations,
see \cite{S3}.
We proposed an alternative approach based on the real Aronhold set,
$\Cal L=\{L_0,\dots,L_6\}$, corresponding to a given typical 7-configuration
$\P$. Namely, the order in which cubic $C_i$ passes through the points $p_j$ is the order in which bitangent $L_i$ intersects other bitangents $L_j$. The two branches of $C_i$ at the node correspond to the two tangency points of $L_i$.

\begin{remark}
Possibility of two shapes of cubic $C_6$ shown on Figure \ref{nodalcubic} correspond to possibility to deform a real quartic with 4 ovals, so that
bitangent $L_6$ moves away from an oval,  as it is shown in the Appendix
on the top Figure:
the two tangency points to $L_6$ on that
oval are deformed into two imaginary (complex conjugate) tangency points.
 Similarly, one can shift double bitangents to the same ovals in the other
 of real Aronhold sets shown in the Appendix.
\end{remark}
\begin{remark}
 The two loops (finite and infinite) of a real nodal cubic $C_i$ that correspond to the two line segments on $L_i$ bounded by the tangency points can be distinguished by the following parity rule.
Line $L_i$ contains six points of intersection with $L_j$, $0\le j\le6$, $j\ne i$, and one more intersection
point, with a line  $L_i'$ obtained by shifting $L_i$ away from the real locus of the quartic.
One of the two line segments contains even number of intersection points, and it
 corresponds to the ``finite'' loop of $C_i$, and the other line segment represents the ``infinite'' loop
of $C_i$.
\end{remark}

\subsection{Method of Cremona transformations} An elementary real Cremona transformation, $\Cr ijk :\Rp2\to\Rp2$,
based at a triple of points $\{p_i,p_j,p_k\}\subset\P$ transforms a typical 7-configuration $\P=\{p_0,\dots,p_6\}$ to another
typical 7-configuration $\P_{ijk}=\Cr{i}jk(\P)$.
Starting with a configuration $\P\in\QC7{(7,0,0,0)}$,
we can realize the other 13 Q-deformation classes of 7-configurations as $\P_{ijk}$ for a suitable choice of $i,j,k$, as it is
shown on Figure \ref{Cremona}, see \cite{Z} for more details.
This construction is used to produce the real Aronhold sets shown in the Appendix.

\begin{figure}[h]
\begin{center}
{\scalebox{0.6}{\includegraphics{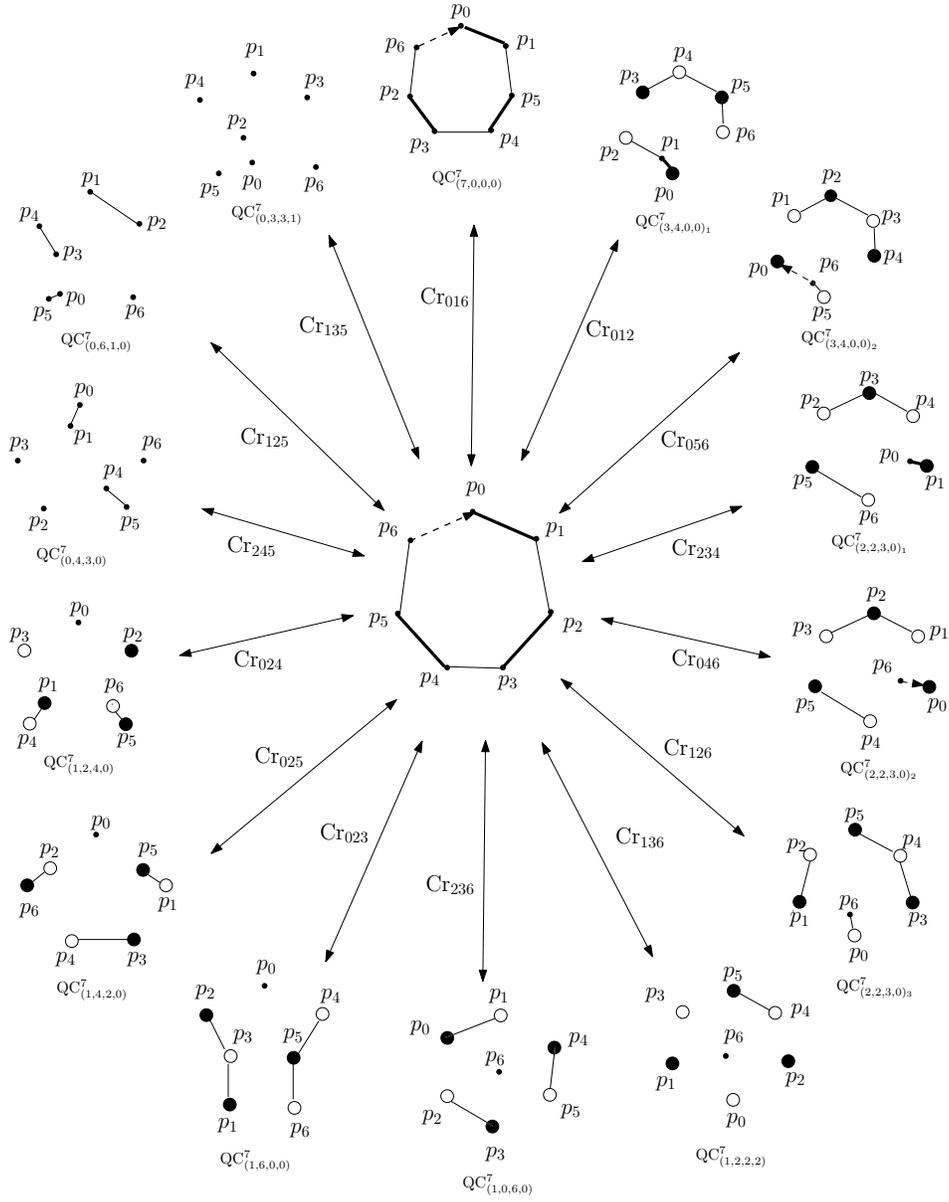}}}
\end{center}
\caption{Cremona transformations of $\mathcal{P}\in QC^{7}_{(7,0,0,0)}$}
\label{Cremona}
\end{figure}

\clearpage
\section*{Appendix. Real Aronhold sets.}
The 14 Figures below show real Aronhold sets representing typical planar $7$-configurations.
In the case of $\QC7{(7,0,0,0)}$ on the top Figure we have shown a possible variation
of one of the bitangents that
has two contacts to the same oval: it can be shifted from this oval after a deformation of the quartic,
so that the contact points become imaginary.
Similar variations are possible in the other 9 cases (except
$\QC7{(3,4,0,0)_2}$, $\QC7{(2,2,3,0)_2}$, $\QC7{(2,2,3,0)_3}$, and $\QC7{(1,2,2,2)}$).

\begin{figure}[h!]
\centering
\subfigure[]{\scalebox{0.4}{\includegraphics{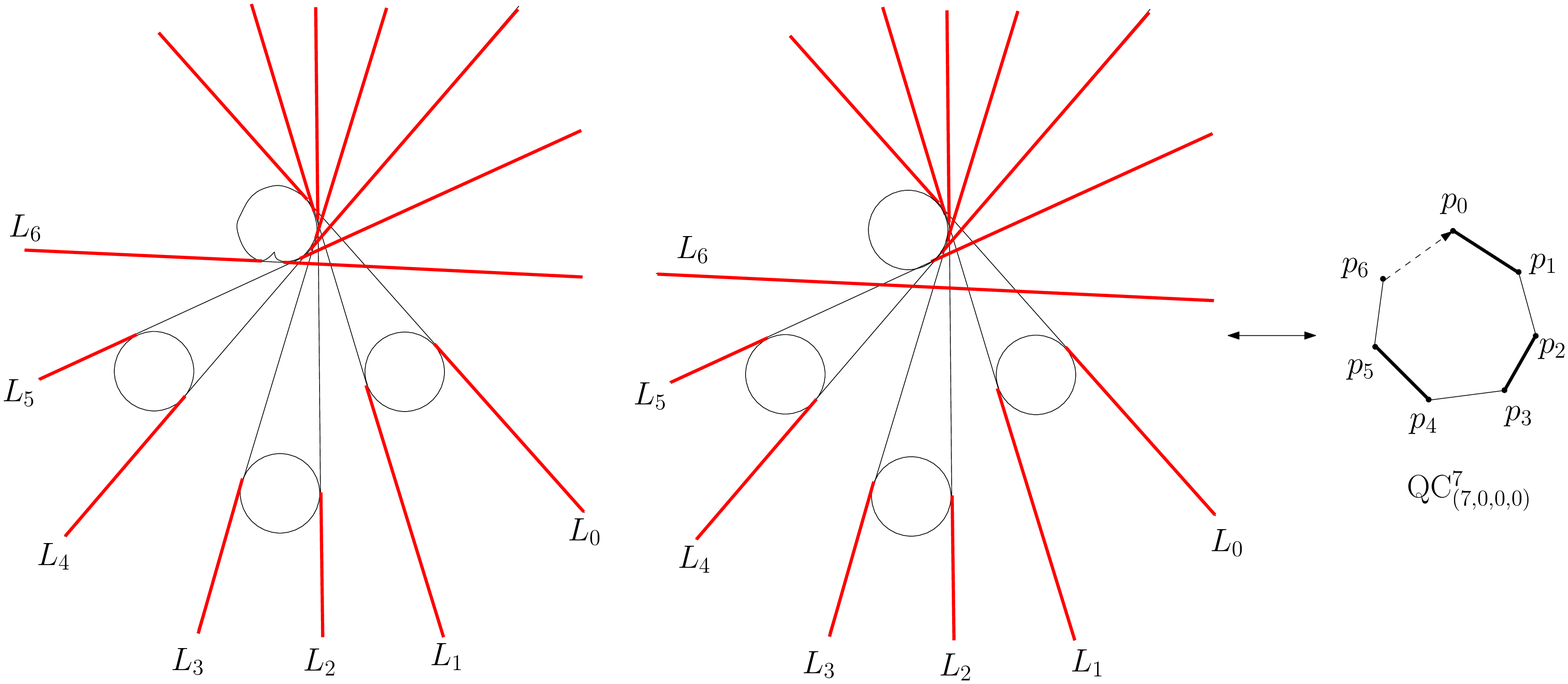}}}\\
\subfigure[$\Cr012$]{\scalebox{0.42}{\includegraphics{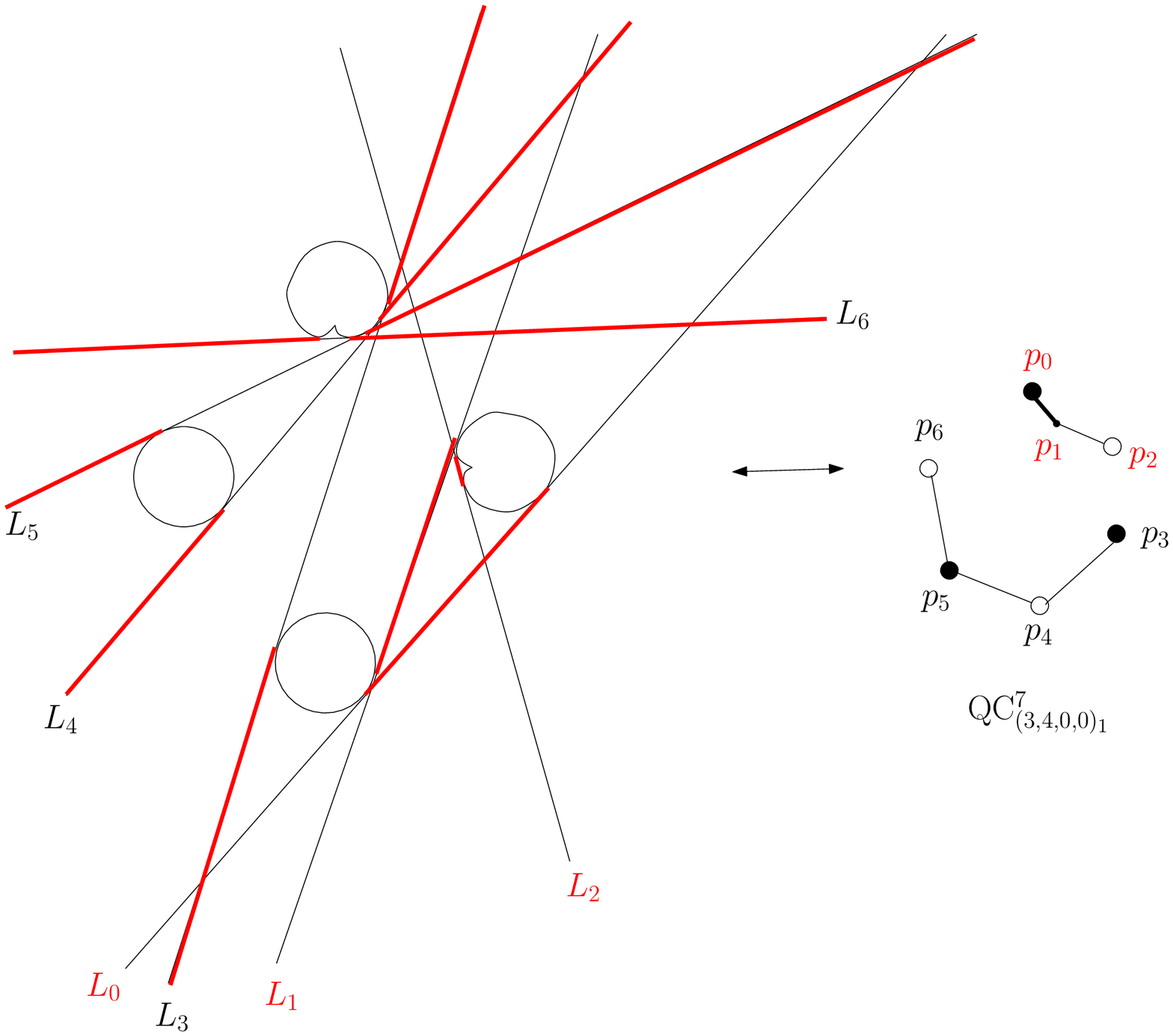}}}
\end{figure}

\begin{figure}[h!]
\ContinuedFloat
\centering
\subfigure[$\Cr056$]{\scalebox{0.58}{\includegraphics{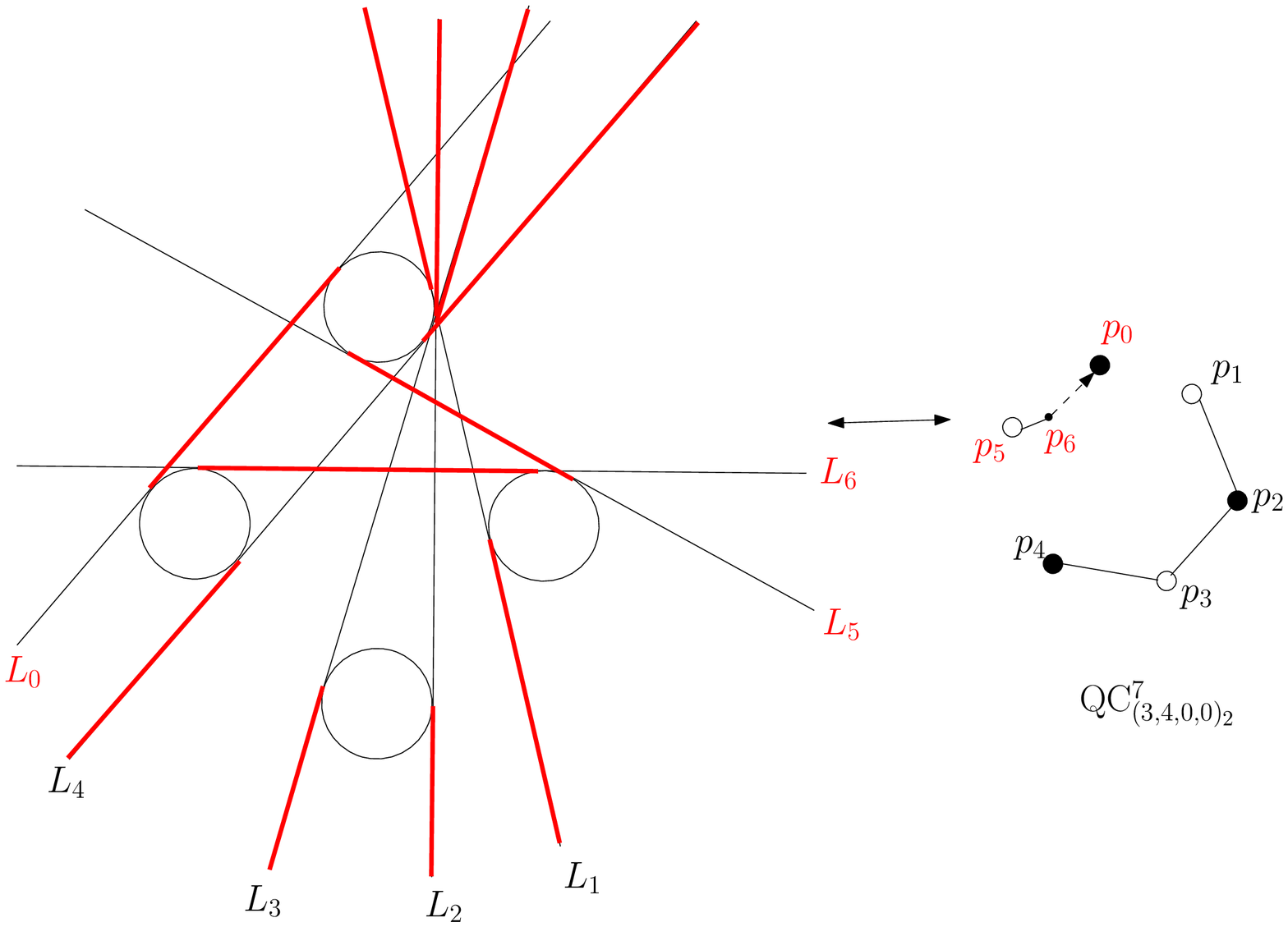}}}\\
\vspace{0.7cm}
\subfigure[$\Cr234$]{\scalebox{0.55}{\includegraphics{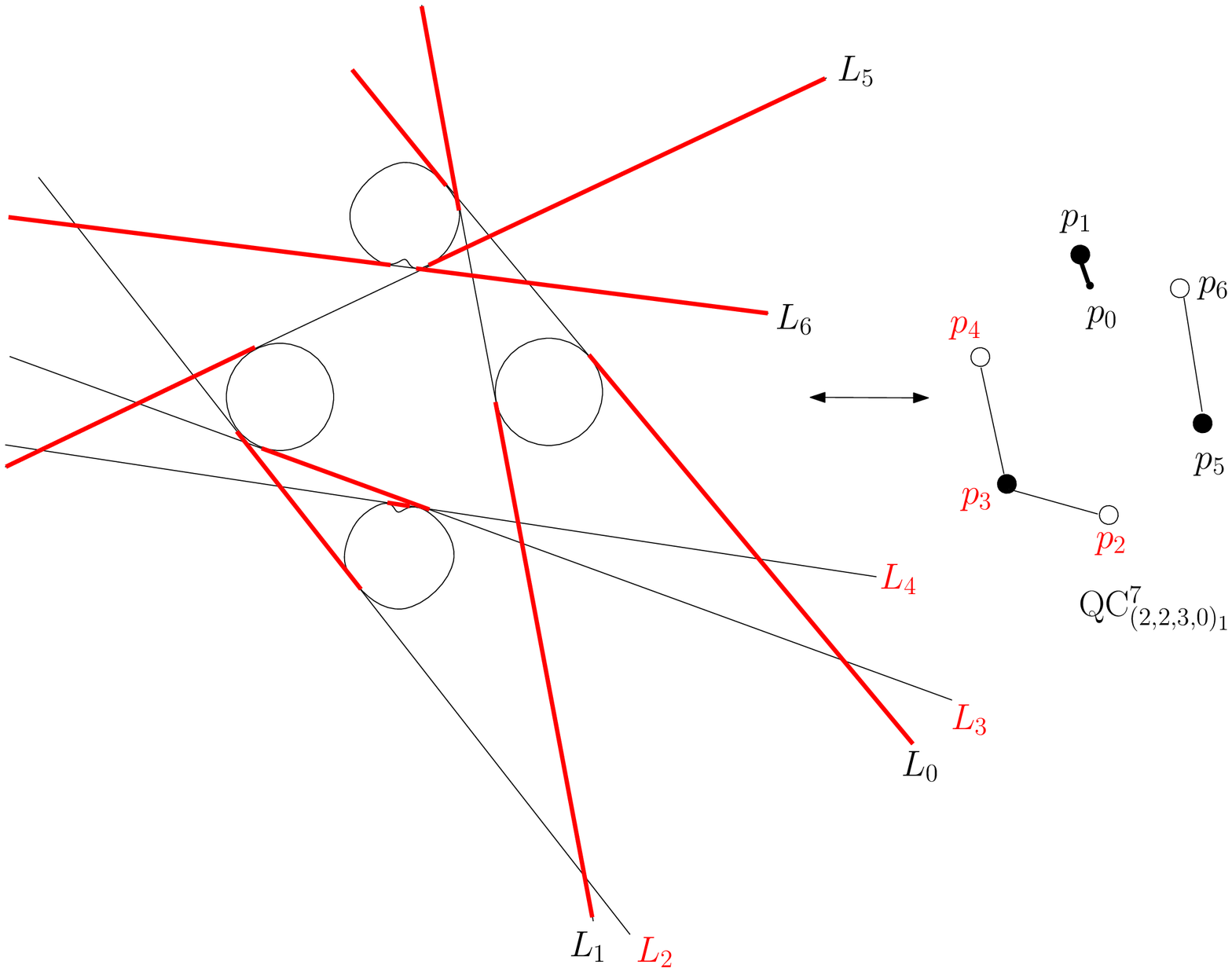}}}
\end{figure}

\begin{figure}[h!]
\ContinuedFloat
\centering
\subfigure[$\Cr046$]{\scalebox{0.6}{\includegraphics{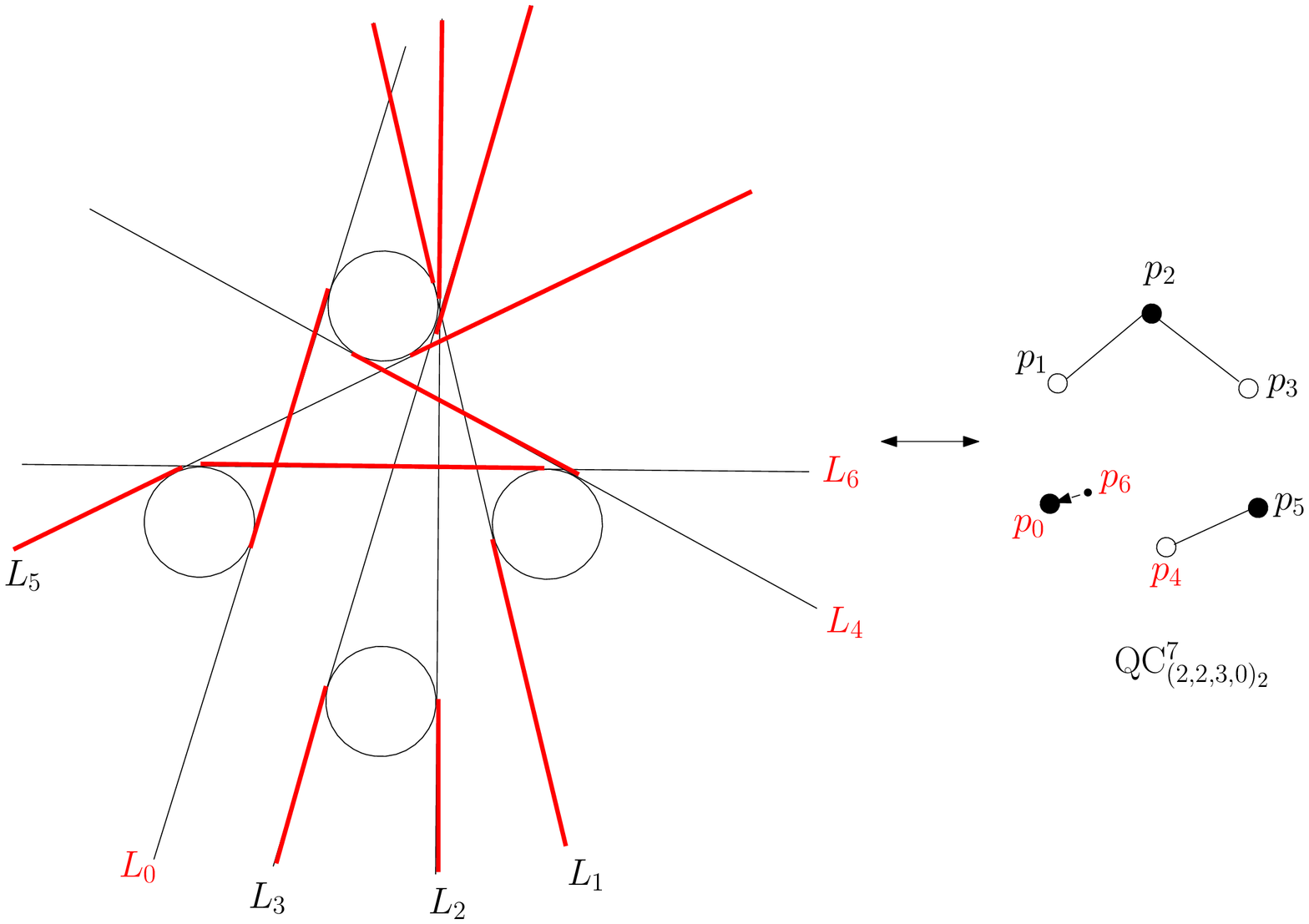}}}\\
\vspace{0.7cm}
\subfigure[$\Cr126$]{\scalebox{0.55}{\includegraphics{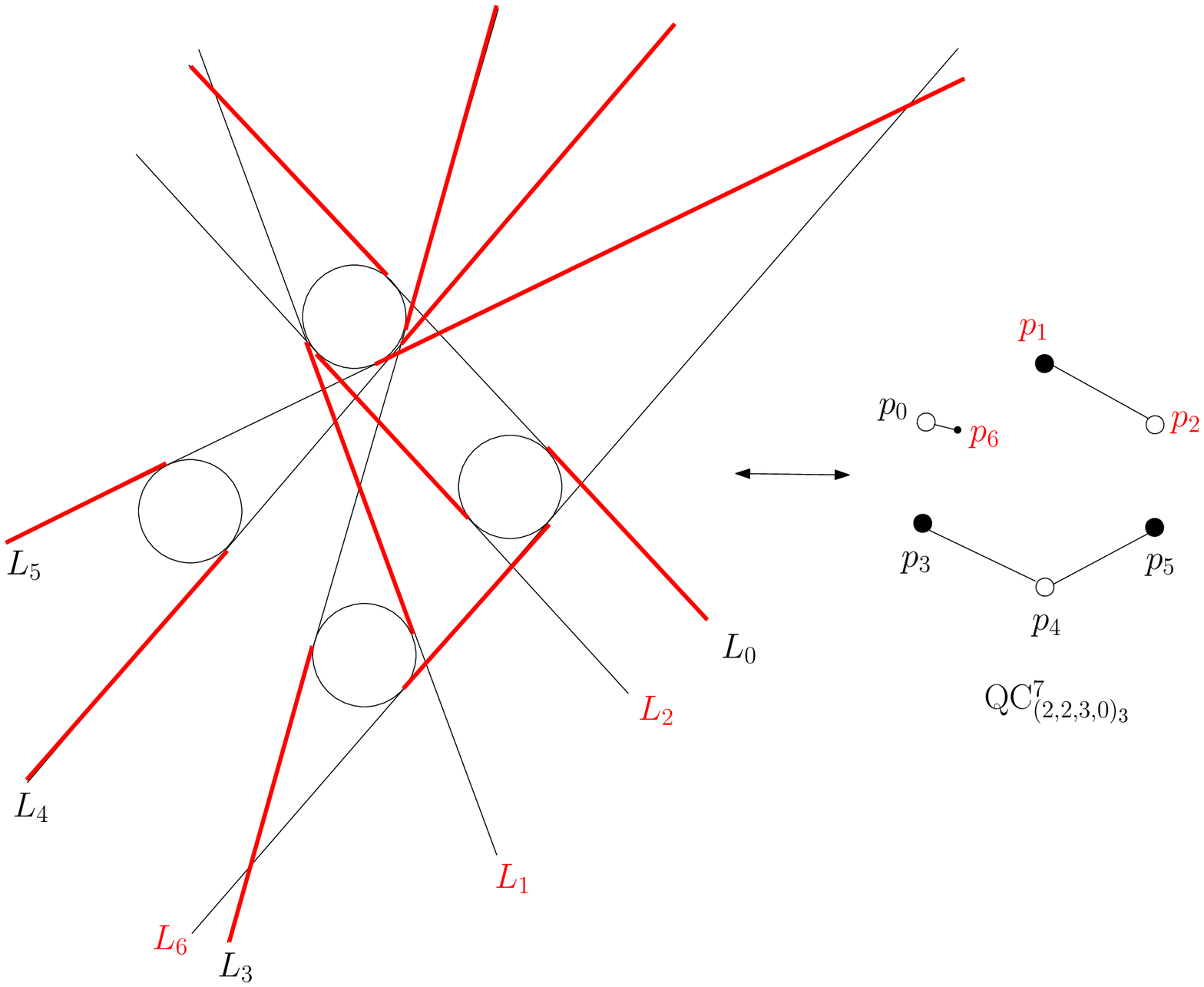}}}
\end{figure}

\begin{figure}[h!]
\ContinuedFloat
\centering
\subfigure[$\Cr136$]{\scalebox{0.55}{\includegraphics{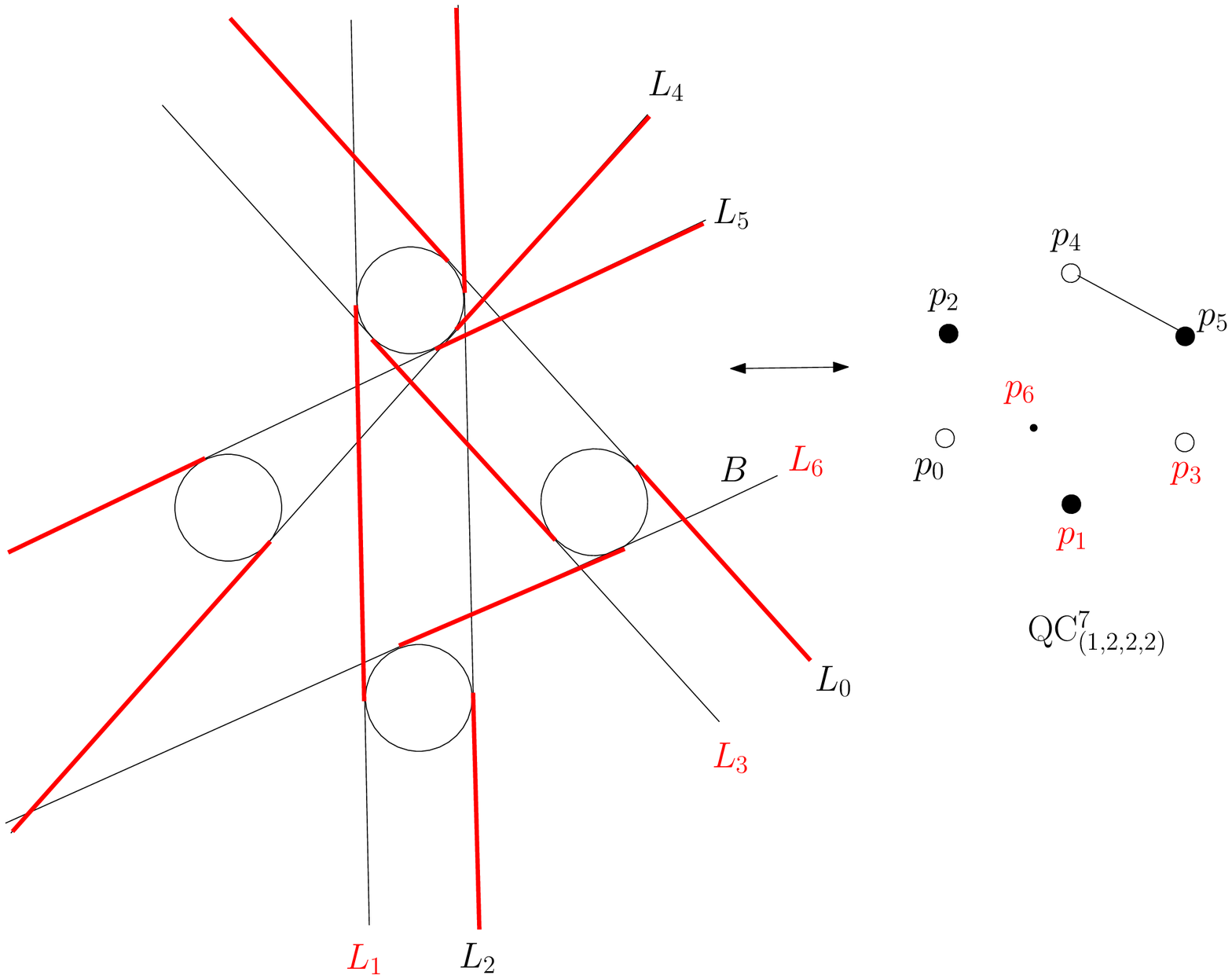}}}\\
\vspace{0.7cm}
\subfigure[$\Cr236$]{\scalebox{0.58}{\includegraphics{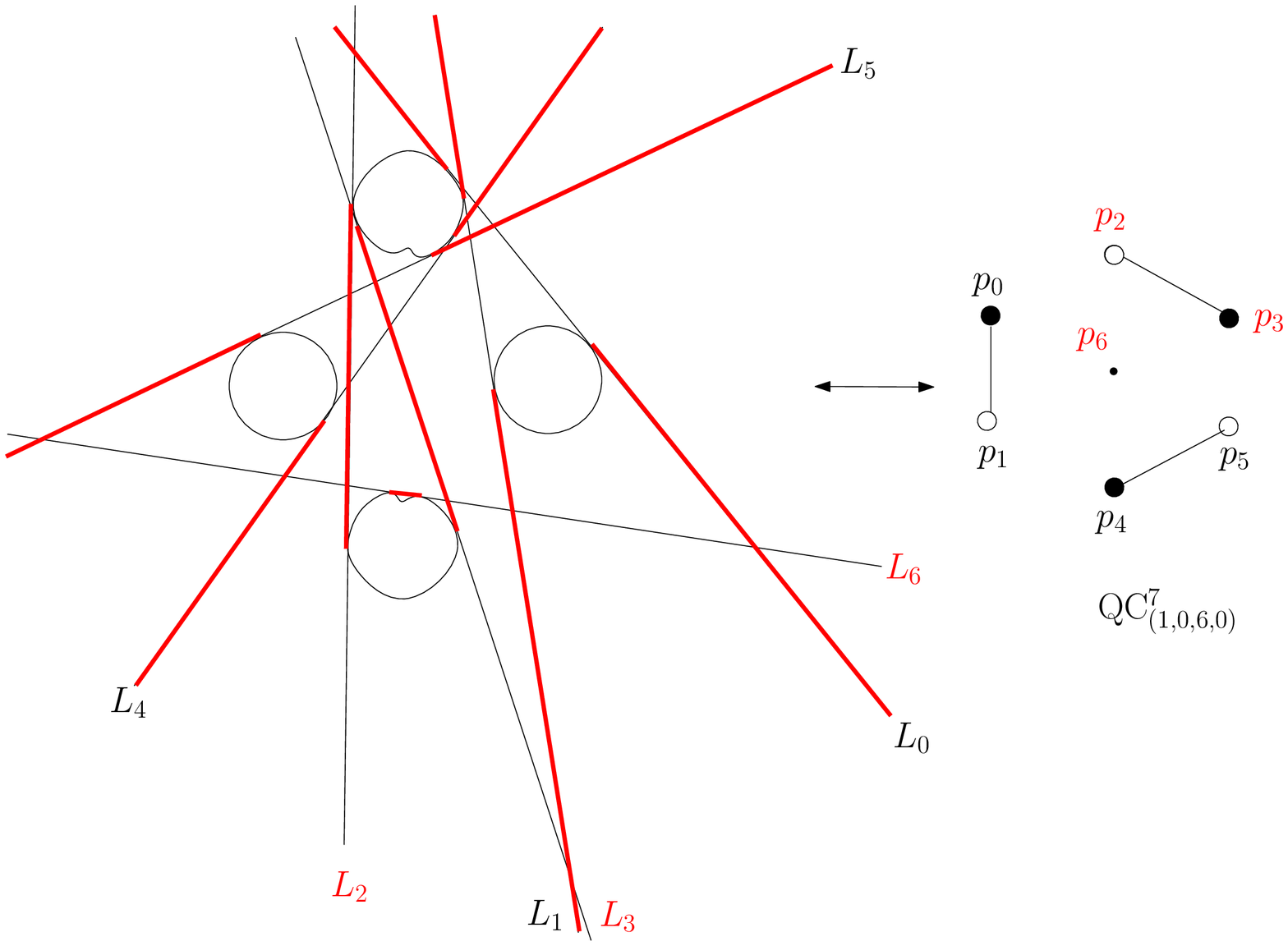}}}
\end{figure}

\begin{figure}[h!]
\ContinuedFloat
\centering
\subfigure[$\Cr023$]{\scalebox{0.58}{\includegraphics{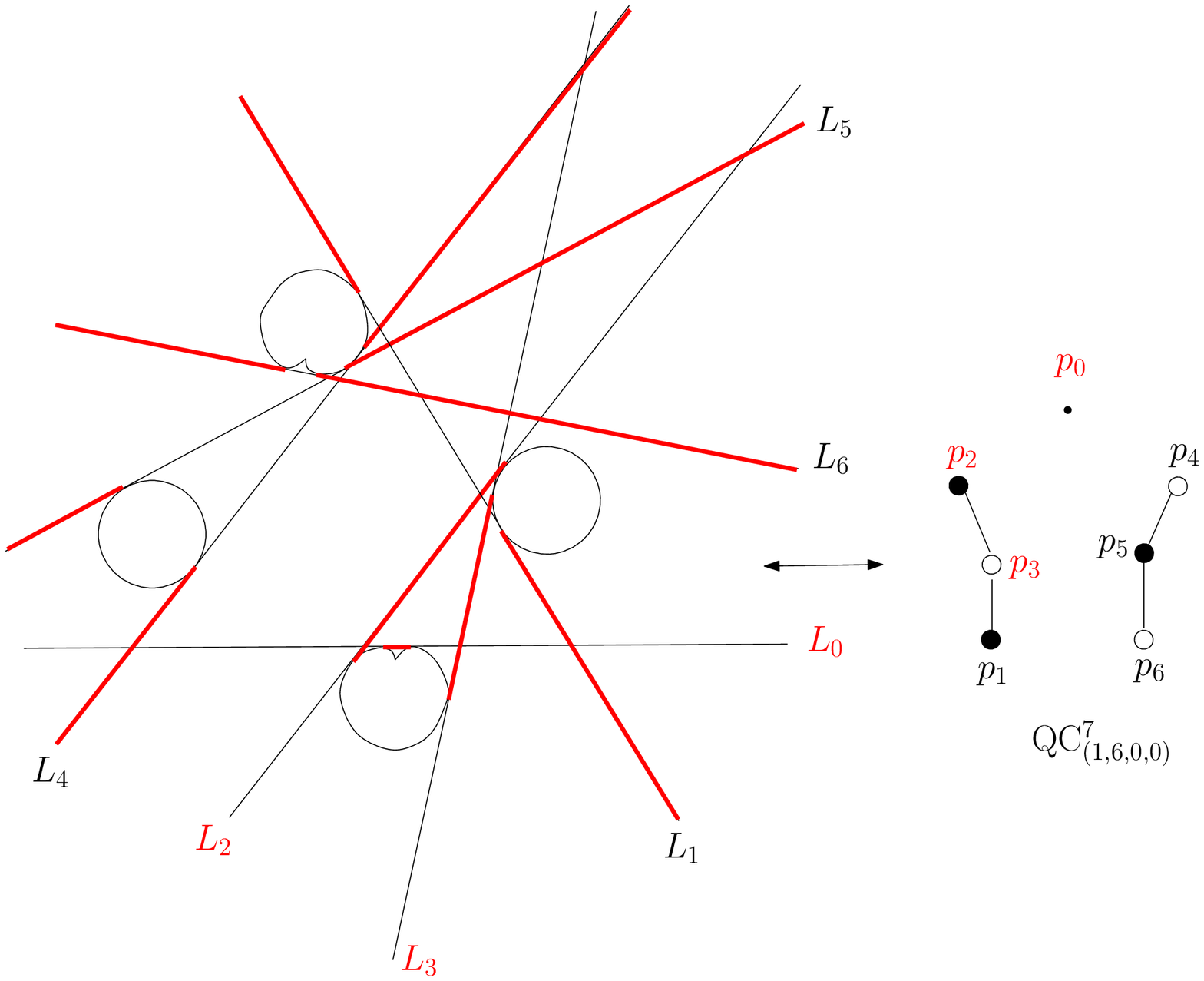}}}\\
\vspace{0.7cm}
\subfigure[$\Cr025$]{\scalebox{0.55}{\includegraphics{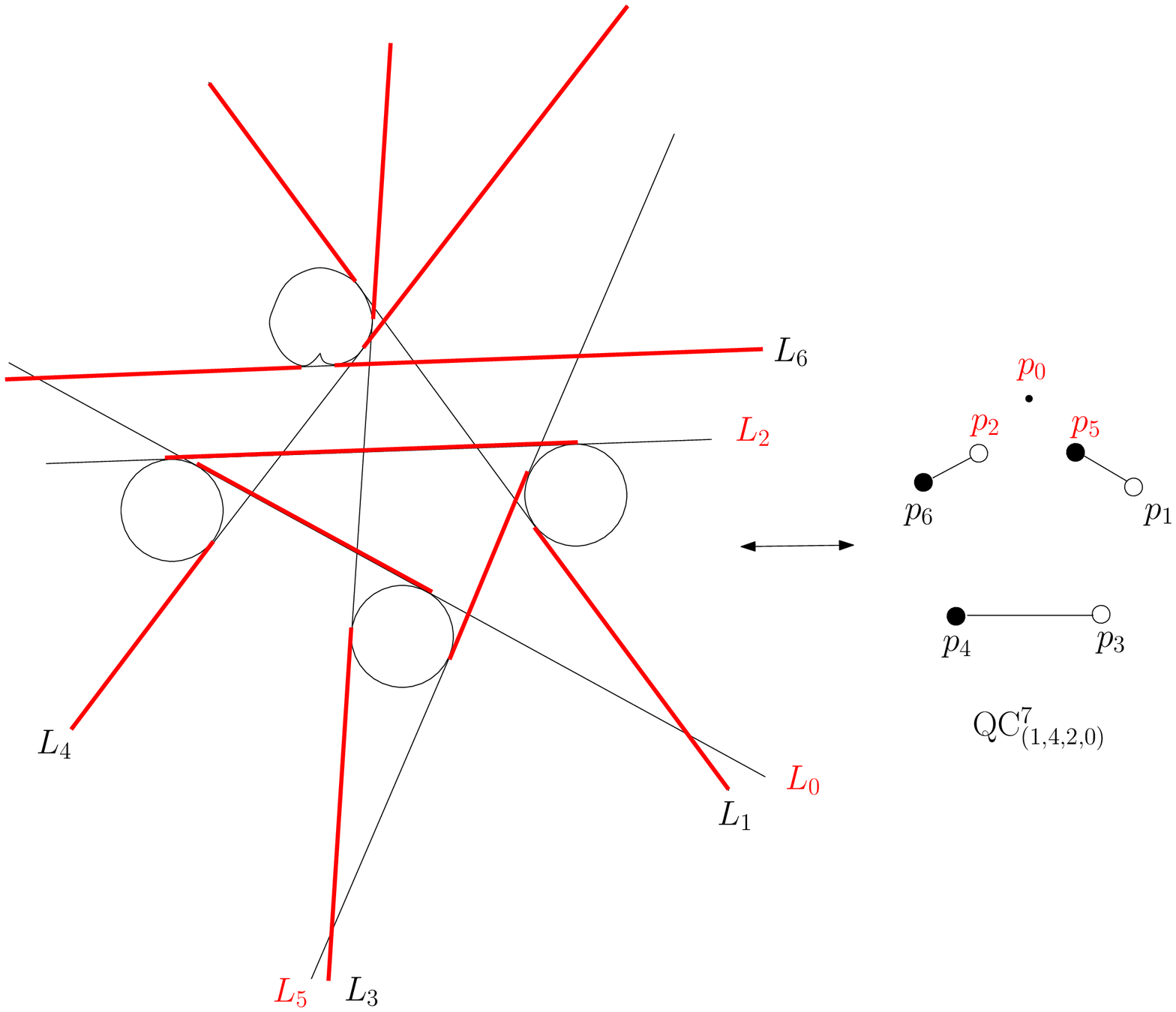}}}
\end{figure}

\begin{figure}[h!]
\ContinuedFloat
\centering
\subfigure[$\Cr024$]{\scalebox{0.55}{\includegraphics{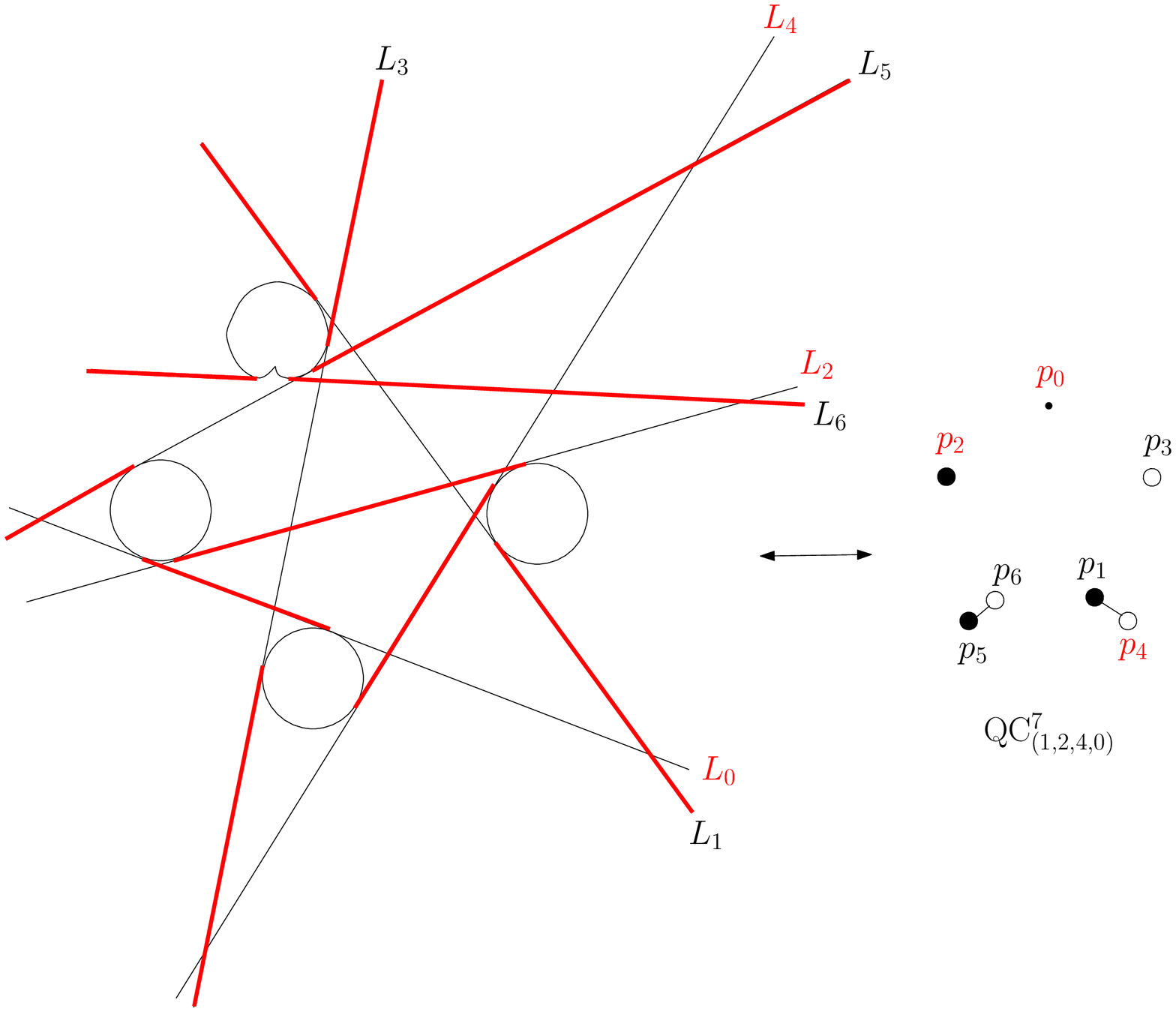}}}\\
\vspace{0.7cm}
\subfigure[$\Cr245$]{\scalebox{0.55}{\includegraphics{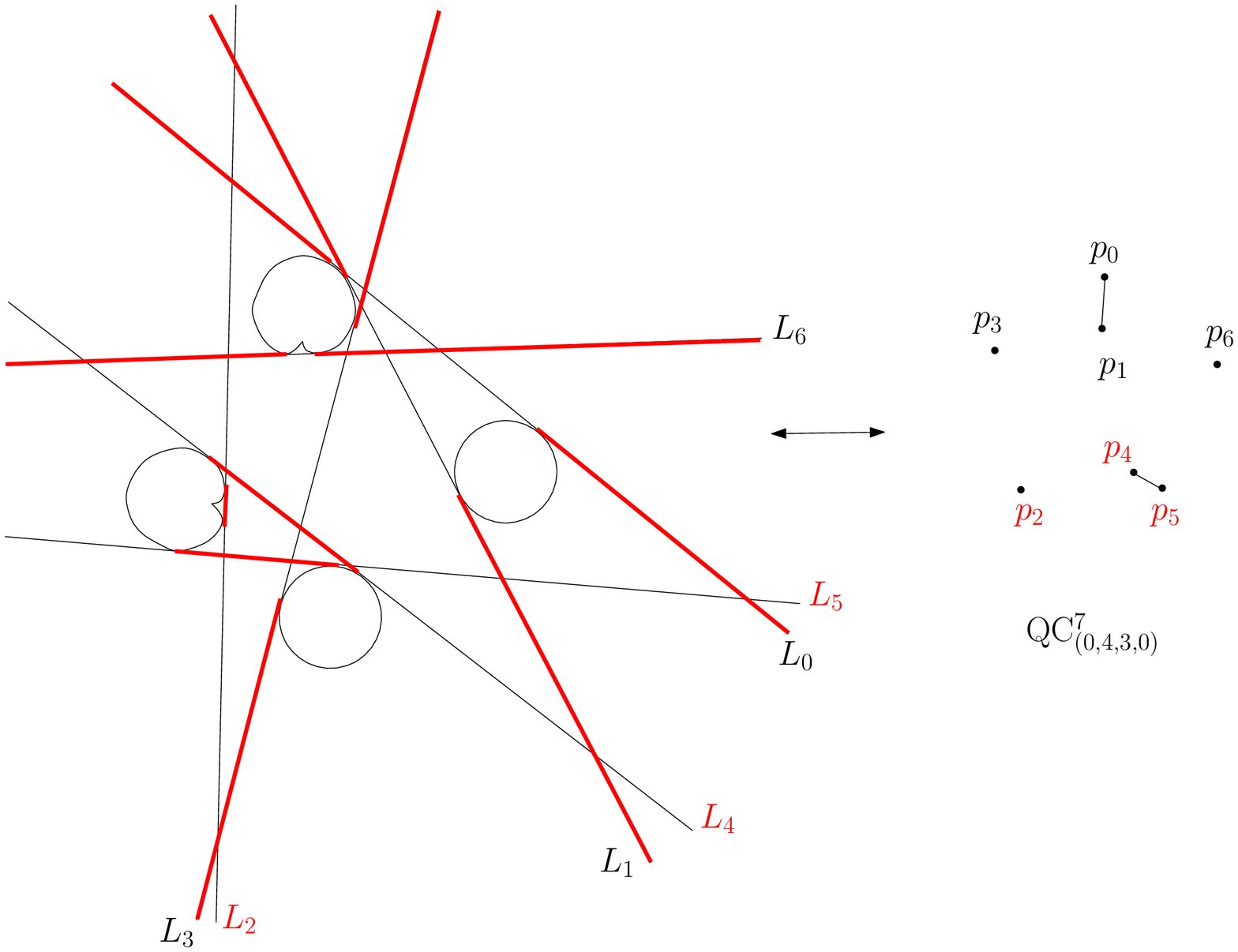}}}
\end{figure}

\begin{figure}[h!]
\ContinuedFloat
\centering
\subfigure[$\Cr125$]{\scalebox{0.55}{\includegraphics{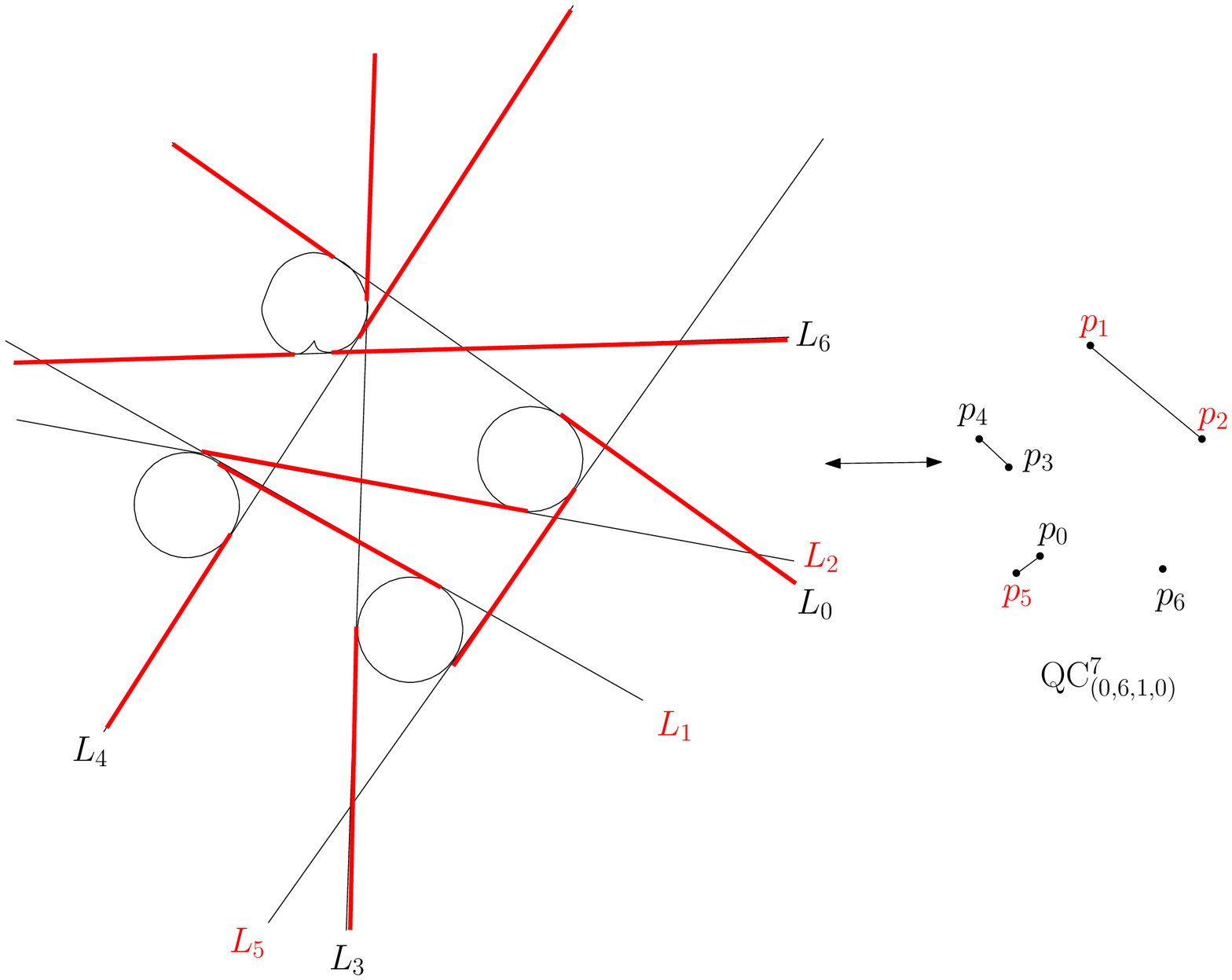}}}\\
\vspace{0.7cm}
\subfigure[$\Cr135$]{\scalebox{0.58}{\includegraphics{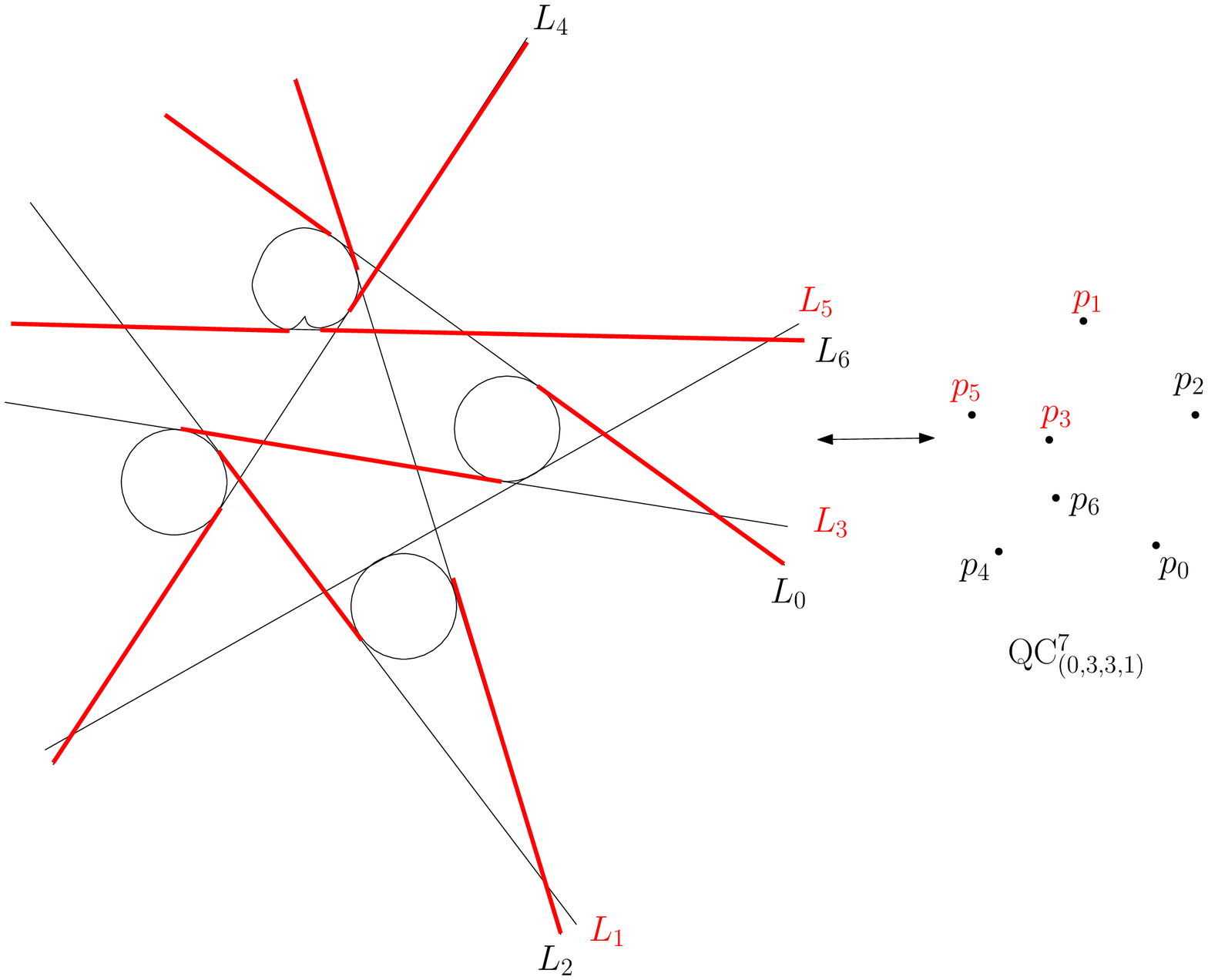}}}
\caption{}
\label{CAC7graph}
\end{figure}

\end{document}